\documentclass{amsart}
\usepackage{amssymb,tikz,amsthm, amsmath}
\usepackage{enumitem}
\newtheorem{thm}{Theorem}

\newtheorem{Prop}[thm]{Proposition}

\newtheorem{Def}[thm]{Definition}
\newtheorem{Ex}[thm]{Example}
\newtheorem{Rmk}[thm]{Remark}

\date{}

\newcommand{\RR}{{\mathbb{R}}}

\begin{document}

\title{Realization of distance matrices by graphs of genus 1}

\author{Cristiano Bocci, Chiara Capresi}

\address{Cristiano Bocci\\
Department of Information Engineering and Mathematics, University of Siena\\
Via Roma, 56 Siena, Italy}
\email{cristiano.bocci@unisi.it}

\address{Chiara Capresi\\
Department of Information Engineering and Mathematics, University of Siena\\
Via Roma, 56 Siena, Italy}
\email{chiara.capresi@yahoo.it}

\begin{abstract} Given a distance matrix $D$, we study the behavior of its compaction vector and reduction matrix with respect to the problem of the realization of $D$ by a weighted graph. To this end, we first give a general result on realization by $n-$cycles and successively  we mainly  focus  on graphs of genus 1, presenting an algorithm which determines when a distance matrix is realizable by such a kind of graph, and then, shows how to construct it.
\end{abstract}



\maketitle

\section{Introduction}

Let $D$ be a matrix whose rows and columns are indexed by a set $X$.
We assume that $D$ is symmetric and has zero entries on the main
diagonal. In phylogenetics, these kind of matrices are called {\it
dissimilarity matrices}. Usually, we take $X=[n]:=\{1, 2, \dots,
n\}$. Hence a dissimilarity matrix $D$ can also be seen as a map
$D:[n]^2\to \RR$, with $D(i,j)=D(j,i)$ and $D(i,i)=0$ for each $i,j\in[n]$. 

A {\it distance matrix} (sometimes also called {\it metric}) is a non-negative dissimilarity matrix which satisfies the
triangle inequality $D(i,j)\leq D(i,k)+D(k,j)$ for all $i,j,k\in X$.  In \cite{Chung, HY}, one can find some results about this kind of matrices.

We say that $D$ has a {\it realization} if there is a weighted graph 
(so one where a non-negative weight is assigned to each edge) 
whose nodes' set contains $X$ and such that the distance (i.e. the length of
the shortest path) between nodes $i,j \in X$ is exactly $D(i,j)$. 
The problem of realizing a distance matrix by a graph of minimal total weight is a difficult problem.  A lot of authors have considered the special case where the graph is a tree and the distance matrix is a tree metric, also called an additive metric. 
This case has
been studied intensively and is well understood. The main
result is the so-called Tree Metric Theorem, (see \cite{Buneman} or \cite[Theorem 2.36]{PaSt}), which is based on  the {\it four-point
condition}, which  is a necessary and sufficient condition for a matrix
to be realized by a tree.

Efficient algorithms constructing such trees were published in \cite{CR} or \cite{SP}. But real data describe merely a tree metric because they arise from a similarity measure that includes errors. Such a measure appears in various fields such as the study of evolution (\cite{MN,A}), the synthesis of certain electrical circuits (\cite{HY2}) or the traffic modelling (\cite{BT}).
Thus, many authors address the problem of realization by a graph which is not a tree. For example in \cite{SP87} and \cite{SP} the author characterizes, respectively, distance matrices which have a unicyclic graph as unique optimal realization and introduces an algorithm for finding optimal graph of non-tree realizable distance matrices. In \cite{BR} can be found results of realization for some particular classes of graphs: paths, caterpillar, cycles, bipartite graphs, complete graphs and planar graphs.  The particular class of cactus metric is studied deeply in \cite{HHMM}.

In this paper we consider the problem of realizing a  distance matrix by a weighted graph of genus 1, i.e. a graph containing  only one cycle. This kind of graphs represents a first non-trivial case in the study of phylogenetic networks. See, for example, Figure 4 in \cite{BHMS}, Figures 2 and 5 in \cite{HMSW}, Figures 1 and 4 in \cite{vIM}.

Using and generalizing the definitions of compaction and reduction processes contained in \cite{V} (and basing, especially compaction, on the results in \cite{HY}), we establish a recursive method to ``contract'' a distance matrix $D$, until we end with a matrix where  contractions are no more possible. The analysis of this final matrix by Theorem \ref{n-agono}, and Propositions \ref{proprank2} and \ref{propstar},  permits to establish if  it is realizable, or not, by an $n-$cycle or a tree. Then, in case of positive answer, moving backward we are able to reconstruct the tree or the graph of genus 1 which realizes $D$.
Moreover, Theorem \ref{n-agono}, combined with a condition on compaction values, furnishes a result on optimality of the realization; this is stated in Proposition \ref{optimal}.
The case of graph of genus 1 can be seen as a particular case of cactus graphs, however we would like to point out that the approach of this paper is different with respect to \cite{HHMM}. The attempt, here, is to give an algorithm that gives an optimal realization of a distance matrix $M$, just working with repeated processes of compaction and reduction on it, instead of starting with a given realization of $M$ and searching for the optimal one.
The computational cost of the algorithm, described in Section \ref{AL}, is $O(n^4)$.

\section{Preliminaries}
We initially recall, from \cite{V}, the definitions of compaction and reduction processes.

Let $D$ be a distance matrix of order $n \times n$ and consider the $n \times n-$matrices $E_i$  where

\[
(E_i)_{jk}=
\begin{cases}
1 & \mbox{if } j=i\not= k\cr
1 & \mbox{if } j\not=i= k\cr
0 & \mbox{elsewhere}
\end{cases}
\]

\noindent Let $D_i(\alpha)=D-\alpha \cdot E_i$. The following result establishes for which values of $\alpha$, $D_i(\alpha)$ is still a distance matrix.

\begin{thm}[\cite{HY}, Lemma 1]\label{thmHak}
$D_i(\alpha)$ is a distance matrix if and only if $\alpha \leq  \frac{1}{2} \cdot (d_{pi}+d_{ir}-d_{pr})$, for all $p,r\not=i$.
\end{thm}

The new matrix $D_i(\alpha)$, obtained from $D$, is called the  {\bf $i-$th compaction matrix of $D$ with respect to $\alpha$}. The compaction of an index $i$ of $D$ leads to a new matrix with, eventually, equal rows (and by symmetry,  equal columns). By deleting all, but one, repeated rows and columns we obtain a new matrix which is called {\bf $i-$th reduction matrix of $D$ with respect to $\alpha$}.

In \cite{HY} and \cite{V} the authors work with any admissible $\alpha$ satisfying the condition of Theorem \ref{thmHak}, but, here, we are interested in considering the maximal value $\alpha$ for which $D_i(\alpha)$ is still a distance matrix. 
Moreover, we want to work for each  index  $i$ of compaction, with $i=1, \dots, n$, and collect these data  into a vector. This leads to the following

\begin{Def}
Given a distance matrix  $D$ of order $n \times n$, the {\bf compaction vector} of $D$, ${\bf a}_D=(a_1, \dots, a_n)$, is defined  as 
\[
a_i = \frac{1}{2} \cdot \min_{p\not=i,r\not= i}\{d_{pi}+d_{ir}-d_{pr}\}.
\]
\end{Def}

\begin{Rmk}\label{rmkleaf}\rm
Consider the tree in Figure \ref{rmkleaffig}. Its distance matrix is
\[
D=\begin{pmatrix}
0 &3 &5 &6\cr
3& 0 & 6 &7\\
5 &6 & 0&7\\
6 &7 &7&0
\end{pmatrix}.
\]
The compaction vector of $D$ is ${\bf a}_D=(1,2,3,4)$.
\begin{figure}
\centering

\begin{tikzpicture}
  [scale=0.8,auto=left]

  \node[circle,draw=black!100,fill=black!20,inner sep=2pt] (n2) at (0,-1) {$2$};
  \node[circle,draw=black!100,fill=black!20,inner sep=2pt]  (n1) at (0,1) {$1$};
  \node[circle,draw=black!100,fill=black!20,inner sep=2pt]  (n3) at (4,1) {$3$};
  \node[circle,draw=black!100,fill=black!20,inner sep=2pt]  (n4) at (4,-1) {$4$};
  \node[circle,draw=black!100,fill=black!20,inner sep=5pt]  (n5) at (1,0) {};
  \node[circle,draw=black!100,fill=black!20,inner sep=5pt]  (n6) at (3,0) {};
 
 \foreach \from/ \to in {n1/n5, n2/n5,n3/n6, n4/n6,n5/n6}
    \draw (\from) ->(\to);

\draw (0.6,0.7) node[scale=1] {$1$};
\draw (0.6,-0.7) node[scale=1] {$2$};
\draw (3.4,0.7) node[scale=1] {$3$};
\draw (3.4,-0.7) node[scale=1] {$4$};
\draw (2,0.3) node[scale=1] {$1$};    
    \end{tikzpicture}
    
    \caption{The tree of Remark \ref{rmkleaf}.}\label{rmkleaffig}
\end{figure}
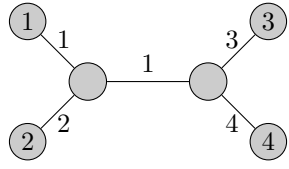
It is a well known fact that, in case $D$ is obtained from a weighted graph and if $i$ is a leaf, then the entry $a_i$ is the weight of the edge connecting the leaf $i$ to its internal node.

On the other hand, consider the graph in Figure \ref{rmkleaffig2}(a). Its distance matrix is
\begin{equation}\label{2rel}
D=\begin{pmatrix}
0 &3 &5 &4\cr
3& 0 & 5 &5\\
5 &5 & 0&5\\
4 &5 &5&0
\end{pmatrix}
\end{equation}
and its compaction vector is ${\bf a}_D=(1,\frac{3}{2},\frac{5}{2},2)$. In this case, the entries of 
${\bf a}_D$ do not correspond to the weights of the edges incident to the leaves. However,  we can notice that the graph in Figure \ref{rmkleaffig2}(b) has the same distance matrix, and hence the same compaction vector, of  the graph in Figure \ref{rmkleaffig2}(a) and, for this graph, the weights of the edges connecting the leaves to the internal nodes correspond to the entries of the compaction vector. Both these graphs are realization of the matrix in (\ref{2rel}).

\begin{figure}
\centering
\begin{tabular}{ccc}
\begin{tikzpicture}
  [scale=0.8,auto=left]

  \node[circle,draw=black!100,fill=black!20,inner sep=2pt] (n1) at (-1,-1) {$1$};
  \node[circle,draw=black!100,fill=black!20,inner sep=2pt]  (n2) at (-1,3) {$2$};
  \node[circle,draw=black!100,fill=black!20,inner sep=2pt]  (n3) at (3,3) {$3$};
  \node[circle,draw=black!100,fill=black!20,inner sep=2pt]  (n4) at (3,-1) {$4$};
  \node[circle,draw=black!100,fill=black!20,inner sep=5pt]  (n5) at (0,0) {};
  \node[circle,draw=black!100,fill=black!20,inner sep=5pt]  (n6) at (2,0) {};
 \node[circle,draw=black!100,fill=black!20,inner sep=5pt]  (n7) at (0,2) {};
  \node[circle,draw=black!100,fill=black!20,inner sep=5pt]  (n8) at (2,2) {};

 \foreach \from/ \to in {n1/n5, n2/n7,n3/n8, n4/n6,n5/n6,n6/n8,n8/n7,n7/n5}
    \draw (\from) ->(\to);

\draw (1,2.3) node[scale=1] {$2$};
\draw (1,-0.3) node[scale=1] {$1$};
\draw (-0.2,1) node[scale=1] {$1$};
\draw (2.2,1) node[scale=1] {$1$};
\draw (-0.65,-0.35) node[scale=1] {$1$};    
\draw (2.65,-0.35) node[scale=1] {$2$};    
\draw (-0.65,2.35) node[scale=1] {$1$};    
\draw (2.65,2.35) node[scale=1] {$2$};    

    \end{tikzpicture}
    &$\quad$ & \begin{tikzpicture}
  [scale=0.8,auto=left]

  \node[circle,draw=black!100,fill=black!20,inner sep=2pt] (n1) at (-1,-1) {$1$};
  \node[circle,draw=black!100,fill=black!20,inner sep=2pt]  (n2) at (-1,3) {$2$};
  \node[circle,draw=black!100,fill=black!20,inner sep=2pt]  (n3) at (3,3) {$3$};
  \node[circle,draw=black!100,fill=black!20,inner sep=2pt]  (n4) at (3,-1) {$4$};
  \node[circle,draw=black!100,fill=black!20,inner sep=5pt]  (n5) at (0,0) {};
  \node[circle,draw=black!100,fill=black!20,inner sep=5pt]  (n6) at (2,0) {};
 \node[circle,draw=black!100,fill=black!20,inner sep=5pt]  (n7) at (0,2) {};
  \node[circle,draw=black!100,fill=black!20,inner sep=5pt]  (n8) at (2,2) {};

 \foreach \from/ \to in {n1/n5, n2/n7,n3/n8, n4/n6,n5/n6,n6/n8,n8/n7,n7/n5}
    \draw (\from) ->(\to);

\draw (1,2.3) node[scale=1] {$2$};
\draw (1,-0.3) node[scale=1] {$1$};
\draw (-0.2,1) node[scale=1] {$\frac12$};
\draw (2.2,1) node[scale=1] {$\frac12$};
\draw (-0.65,-0.35) node[scale=1] {$1$};    
\draw (2.65,-0.35) node[scale=1] {$2$};    
\draw (-0.65,2.3) node[scale=1] {$\frac32$};    
\draw (2.65,2.3) node[scale=1] {$\frac52$};    

    \end{tikzpicture}\\
(a)&& (b)
    \end{tabular}
    \caption{The graphs of  second part of Remark \ref{rmkleaf}.}\label{rmkleaffig2}
\end{figure}
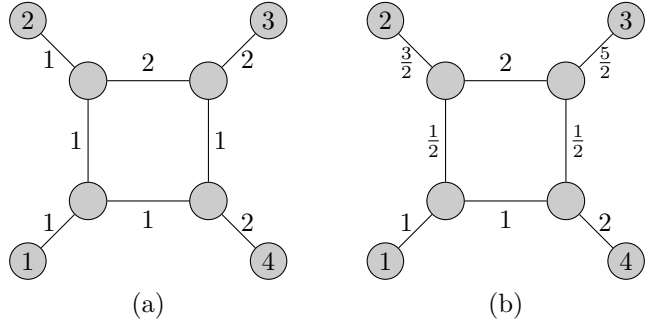

\end{Rmk}

Once the compaction vector  ${\mathbf a}_D$ of $D$ is computed, we consider the following matrix
\begin{equation}\label{cm}
D({\mathbf a}_D) = D - a_1 \cdot E_1 - \cdots - a_n \cdot E_n
\end{equation}

\begin{Def}
The matrix $D({\mathbf a}_D)$ in {\rm (\ref{cm})} is called the {\bf compaction matrix of $D$}.
\end{Def}

\begin{Rmk}\rm Notice that now we omit the expression   ``with respect to  {\bf a}$_D$'' and we do not refer to a specific index $i$, since we are considering the compaction process along all the indices and only with respect to the corresponding values of the compaction vector; while in the standard definition  of compaction  matrix the computation is taken only with respect to an index and a given value $\alpha$ (\cite{V}).
\end{Rmk}

\begin{Def}\label{defrm}
Let $D({\mathbf a}_D)$ be the compaction matrix of $D$ as defined in {\rm(\ref{cm})}. The matrix $\hat{D}_{{\mathbf a}_D}$ obtained removing from $D({\mathbf a}_D)$ all but one repeated rows and columns is called  {\bf reduction matrix of} $D$.
\end{Def}

\begin{Rmk}\rm
Our process of reduction is different with respect to the one in \cite{V}, where only one of the equal rows (and columns) is removed.
\end{Rmk}

\begin{Ex}\label{ex2}\rm
Let $D$ be the following distance matrix
\[
D=\begin{pmatrix}
0 & 4 & 6 & 6 & 3\\
4 & 0 & 5 & 5 & 5\\
6 & 5 & 0 & 2 & 5\\
6 & 5 & 2 & 0 & 5\\
3 & 5 & 5 & 5 & 0
\end{pmatrix}\]
Its compaction vector is
\[
{\mathbf a}_D=\left( 1, \frac{3}{2}, 1, 1, 1\right)
\]
from which we get
\[
D({\mathbf a}_D)=D-E_1 -\frac{3}{2}E_2-E_3-E_4-E_5=
\begin{pmatrix}
0 & \frac{3}{2} & 4 &4 &1\\
\frac{3}{2} & 0 & \frac{5}{2} &\frac{5}{2} &\frac{5}{2}\\
4 & \frac{5}{2} &0 & 0 & 3\\
4 & \frac{5}{2} &0 & 0 & 3\\
1 & \frac{5}{2} & 3 &3& 0
\end{pmatrix}.
\]
Notice that the third and fourth rows of $D({\mathbf a}_D)$ are equal, so the reduction matrix is
\[
\hat{D}_{{\mathbf a}_D}= 
\begin{pmatrix}
0 & \frac{3}{2} &4 &1\\
\frac{3}{2} & 0 & \frac{5}{2}  &\frac{5}{2}\\
4 & \frac{5}{2} &0  & 3\\
1 & \frac{5}{2}  &3& 0
\end{pmatrix}.
\]
\end{Ex}

\begin{Ex}\label{ex1}\rm
Let $D$ be the following distance matrix
\[
D=\begin{pmatrix}
0 & 5 & 7 & 7\\
5 & 0 & 7 & 10\\
7 & 7 & 0 & 9\\
7 & 10 & 9 & 0
\end{pmatrix}.\]
Its compaction vector is
\[
{\mathbf a}_D=\left( 1, \frac{5}{2},3 ,\frac{9}{2}\right)
\]
from which we get
\[
D({\mathbf a}_D)=D-E_1 -\frac{5}{2}E_2-3E_3-\frac{9}{2}E_4=
\begin{pmatrix}
0 &\frac{3}{2} & 3 &\frac{3}{2} \\
\frac{3}{2} &0 &\frac{3}{2} & 3\\
3 & \frac{3}{2}  & 0 & \frac{3}{2} \\
\frac{3}{2}  & 3 & \frac{3}{2} &0
\end{pmatrix}.
\]
Since there are not  equal rows (and columns) in $D({\mathbf a}_D)$,  we conclude that $\hat{D}_{{\mathbf a}_D}= D({\mathbf a}_D)$.
\end{Ex}

We prove now two results concerning particular behavior of compaction and reduction matrices.

\begin{Prop}\label{proprank2}
Let $D$ be a distance matrix of order $n \times n$, if its reduction matrix $\hat{D}_{{\mathbf a}_D}$ has order 2, then $D$ is realizable by a tree.
\end{Prop}

\begin{proof}
If $\hat{D}_{{\mathbf a}_D}$ has order 2, then 
\begin{equation}\label{eqr2}
\hat{D}_{{\mathbf a}_D}
=\begin{pmatrix}
0 & \delta\\
\delta & 0
\end{pmatrix}
\end{equation}
for some $\delta\in \RR_{>0}$.

Moreover, for reduction process, there exist positive integers  $n_1$ and $n_2$, with $n_1+n_2=n$, such that $D({\mathbf a}_D)$ has a set of $n_1$ equal rows (and correspondent columns) and a set of  $n_2$  equal rows (and correspondent columns). For simplicity we assume that the first $n_1$ rows are equal, and hence the remaining $n_2$ rows are equal.

Then by (\ref{eqr2}) we get that $D({\mathbf a}_D)$ has the form
\begin{equation}\label{eqr2_2}
\left(\begin{array}{c|c}
{\mathbf 0_{n_2\times n_1}} & \delta{\mathbf 1_{n_2\times n_2}} \\
\hline
\delta{\mathbf 1_{n_1\times n_1}} & {\mathbf 0_{n_1\times n_2}}
\end{array}\right)
\end{equation}
where ${\mathbf 0_{k\times l}}$ denotes the null matrix of order $k\times l$ and  $\delta{\mathbf 1_{k\times l}}$  denote the matrix of order $k\times l$ whose entries are all equals to $\delta$.

Let  {\bf a}$_D = (a_1,\dots, a_n)$ be the compaction vector of $D$, by definition of compaction matrix we know that
\begin{equation}\label{eqr2_3}
D({\mathbf a}_D)=\begin{pmatrix}
0 & d_{12}-a_1-a_2 & \cdots & d_{1n}-a_1-a_n\\
d_{21}-a_1-a_2 & 0 &  \cdots & d_{2n}-a_2-a_n\\
\vdots & \vdots & \ddots & \vdots\\
d_{n1}-a_1-a_n &   d_{n2}-a_2-a_n &  \cdots &0
\end{pmatrix}
\end{equation}

Comparing the entries of the matrices in (\ref{eqr2_2}) and (\ref{eqr2_3}) one has
\[
d_{ij}=
\begin{cases}
0 & \mbox{if } i = j\\
a_i+a_j & \mbox{if } 1\leq i <j \leq n_1 \mbox{ or }  n_1+1\leq i <j \leq n_1+n_2\\
\delta+a_i+a_j & \mbox{otherwise}
\end{cases}
\]
from which follows that $D$ has a realization by the tree in Figure \ref{treeprop1}.
\end{proof}

\begin{Rmk} We want to point out that, if we reduce one more time the matrix  $\hat{D}_{{\mathbf a}_D}$, we end up with a single point, and this is a well-known characteristic of a tree metric.
\end{Rmk}

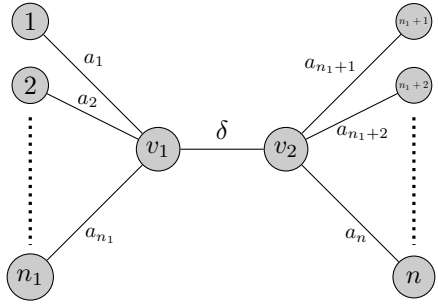
\begin{figure}
\centering
\begin{tikzpicture}
  [scale=0.85,auto=left]
  \node[circle,draw=black!100,fill=black!20,inner sep=2pt] (n1) at (1,1) {$v_1$};
  \node[circle,draw=black!100,fill=black!20,inner sep=2pt]  (n2) at (3,1) {$v_2$};
  \node[circle,draw=black!100,fill=black!20,inner sep=2pt] (n3) at (-1,3) {$1$};
  \node[circle,draw=black!100,fill=black!20,inner sep=2pt]  (n4) at (-1,2) {$2$};
  \node[circle,draw=black!100,fill=black!20,inner sep=2pt] (n5) at (-1,-1) {$n_1$};
  \node[circle,draw=black!100,fill=black!20,inner sep=2pt,scale=0.4]  (n6) at (5,3) {$n_1+1$}; 
  \node[circle,draw=black!100,fill=black!20,inner sep=2pt,scale=0.4]  (n7) at (5,2) {$n_1+2$}; 
    \node[circle,draw=black!100,fill=black!20,inner sep=3pt] (n8) at (5,-1) {$n$};
 \foreach \from/ \to in {n1/n2, n1/n3,n1/n4,n1/n5,n2/n6, n2/n7,n2/n8}
    \draw (\from) ->(\to);

\draw (2,1.3) node[scale=1] {$\delta$};
\draw (0,2.4) node[scale=0.8] {$a_1$};
\draw (-0.1,1.75) node[scale=0.8] {$a_2$};
\draw (0.1,-0.35) node[scale=0.8] {$a_{n_1}$};

\draw (3.7,2.3) node[scale=0.8] {$a_{n_1+1}$};
\draw (4.2,1.25) node[scale=0.8] {$a_{n_1+2}$};
\draw (4.1,-0.35) node[scale=0.8] {$a_n$};

\draw[very thick, dotted] (-1,1.5) -- (-1, -0.5);
\draw[very thick, dotted] (5,1.5) -- (5, -0.5);
\end{tikzpicture}
\caption{Realization for Proposition \ref{proprank2}.}\label{treeprop1}
\end{figure}

\begin{Ex}\label{ex3}\rm Let $D$ be the distance matrix
\[
D=\begin{pmatrix}
0 & 2 & 3 & 3\\
2 & 0 & 3 & 3\\
3 & 3 &0 &2\\
3 & 3 &2 &0\\
\end{pmatrix}\]
From its compaction vector
\[
{\mathbf a}_D=\left( 1,1,1,1\right)
\]
we get
\[
D({\mathbf a}_D)=D-E_1 -E_2-3E_3-E_4=
\begin{pmatrix}
0 & 0& 1 & 1\\
0 & 0& 1 & 1\\
1 & 1 & 0 & 0\\
1 & 1 & 0 & 0
\end{pmatrix}.
\]
In this case $D({\mathbf a}_D)$ has two pairs of equal rows and, after the reduction process, we get
\[
\hat{D}_{{\mathbf a}_D}=\begin{pmatrix}
 0& 1 \\
1 & 0 
\end{pmatrix}.
\]
Then, by Proposition \ref{proprank2}, $D$ has a realization by the tree in Figure \ref{albex3}.

\begin{figure}
\centering
\begin{tikzpicture}
  [scale=0.85,auto=left]
  \node[circle,draw=black!100,fill=black!20,inner sep=2pt] (n1) at (1,1) {$v_1$};
  \node[circle,draw=black!100,fill=black!20,inner sep=2pt]  (n2) at (3,1) {$v_2$};
  \node[circle,draw=black!100,fill=black!20,inner sep=2pt] (n3) at (0,2) {$1$};
  \node[circle,draw=black!100,fill=black!20,inner sep=2pt]  (n4) at (0,0) {$2$};
  \node[circle,draw=black!100,fill=black!20,inner sep=2pt] (n5) at (4,2) {$3$};
  \node[circle,draw=black!100,fill=black!20,inner sep=2pt]  (n6) at (4,0) {$4$}; 
 \foreach \from/ \to in {n1/n2, n1/n3,n1/n4,n2/n5,n2/n6}
    \draw (\from) ->(\to);

\draw (2,1.3) node[scale=1] {$1$};
\draw (0.65,1.65) node[scale=1] {$1$};
\draw (0.65,0.35) node[scale=1] {$1$};
\draw (3.35,1.65) node[scale=1] {$1$};
\draw (3.35,0.35) node[scale=1] {$1$};

\end{tikzpicture}
\caption{Tree which realizes the matrix of Example \ref{ex3}.}\label{albex3}
\end{figure}
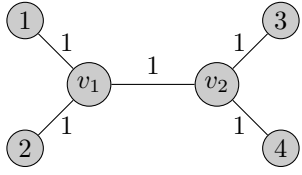

\end{Ex}

The second result concerns the compaction matrix.

\begin{Prop}\label{propstar}
Let $D$ be a distance matrix of order $n \times n$, if its compaction matrix $D({\mathbf a}_D)$ is the null matrix, then $D$ is realizable by a star tree with $n$ leaves (and a unique internal node of valency $n$).
\end{Prop}

\begin{proof}
Let  {\bf a}$_D = (a_1,\dots, a_n)$ be the  compaction vector of $D$. By definition of compaction matrix we know that
\[
D({\mathbf a}_D)=\begin{pmatrix}
0 & d_{12}-a_1-a_2 & \cdots & d_{1n}-a_1-a_n\\
d_{21}-a_1-a_2 & 0 &  \cdots & d_{2n}-a_2-a_n\\
\vdots & \vdots & \ddots & \vdots\\
d_{n1}-a_1-a_n &   d_{n2}-a_2-a_n &  \cdots &0
\end{pmatrix}
\]
Since by hypothesis, $D({\mathbf a}_D)$ is the null matrix, one has
\[
d_{ij}=
\begin{cases}
a_i+a_j & \mbox{if } i\not= j\\
0 & \mbox{otherwise}
\end{cases}
\]
from which easily follows that $D$ has a realization by the star tree in Figure \ref{treeprop2}.
\end{proof}

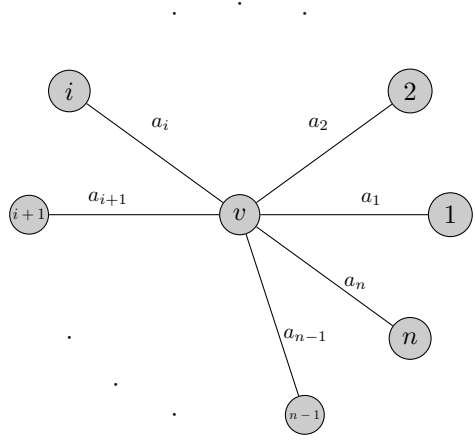
\begin{figure}
\centering
\begin{tikzpicture}
  [scale=0.7,auto=left]
  \node[circle,draw=black!100,fill=black!20,inner sep=3pt] (n1) at (0,0) {$v$};
  \node[circle,draw=black!100,fill=black!20,inner sep=3pt] (n2) at (0:4cm) {$1$};
  \node[circle,draw=black!100,fill=black!20,inner sep=3pt]  (n3) at (36:4cm) {$2$};
  \node[circle,draw=black!100,fill=black!20,inner sep=3pt] (n4) at (144:4cm) {$i$};
    \node[circle,draw=black!100,fill=black!20,inner sep=1pt,,scale=0.6] (n5) at (180:4cm) {$i+1$};
  \node[circle,draw=black!100,fill=black!20,inner sep=2pt,scale=0.5]  (n6) at (288:4cm) {$n-1$}; 
  \node[circle,draw=black!100,fill=black!20,inner sep=3pt]  (n7) at (324:4cm) {$n$}; 
  \node[circle,draw=black!00,fill=black!0,inner sep=3pt]  (n8) at (72:4cm) {$\cdot$}; 
  \node[circle,draw=black!00,fill=black!0,inner sep=3pt]  (n9) at (90:4cm) {$\cdot$}; 
  \node[circle,draw=black!00,fill=black!0,inner sep=3pt]  (n10) at (108:4cm) {$\cdot$}; 
  \node[circle,draw=black!00,fill=black!0,inner sep=3pt]  (n11) at (216:4cm) {$\cdot$}; 
  \node[circle,draw=black!00,fill=black!0,inner sep=3pt]  (n12) at (234:4cm) {$\cdot$}; 
  \node[circle,draw=black!00,fill=black!0,inner sep=3pt]  (n13) at (252:4cm) {$\cdot$}; 

 \foreach \from/ \to in {n1/n2,n1/n3,n1/n4,n1/n5,n1/n6,n1/n7}
    \draw (\from) ->(\to);

\draw (2.5,0.3) node[scale=0.8] {$a_1$};
\draw (1.5,1.73) node[scale=0.8] {$a_2$};
\draw (-1.5,1.73) node[scale=0.8] {$a_i$};
\draw (-2.5,0.3) node[scale=0.8] {$a_{i+1}$};

\draw (1.25,-2.3) node[scale=0.8] {$a_{n-1}$};
\draw (2.2,-1.3) node[scale=0.8] {$a_n$};

\end{tikzpicture}
\caption{Realization for Proposition \ref{propstar}.}\label{treeprop2}
\end{figure}

\begin{figure}
\centering
\begin{tikzpicture}
  [scale=0.85,auto=left]
  \node[circle,draw=black!100,fill=black!20,inner sep=2pt] (n1) at (2,1) {$v$};
  \node[circle,draw=black!100,fill=black!20,inner sep=2pt] (n2) at (0,2) {$1$};
  \node[circle,draw=black!100,fill=black!20,inner sep=2pt]  (n3) at (0,0) {$2$};
  \node[circle,draw=black!100,fill=black!20,inner sep=2pt] (n4) at (4,2) {$3$};
  \node[circle,draw=black!100,fill=black!20,inner sep=2pt]  (n5) at (4,0) {$4$}; 
 \foreach \from/ \to in {n1/n2, n1/n3,n1/n4,n1/n5}
    \draw (\from) ->(\to);

\draw (1,1.75) node[scale=1] {$1$};
\draw (1,0.75) node[scale=1] {$2$};
\draw (3,1.75) node[scale=1] {$3$};
\draw (3,0.75) node[scale=1] {$4$};

\end{tikzpicture}
\caption{Star tree which realizes the matrix of Example \ref{ex4}.}\label{albex4}
\end{figure}
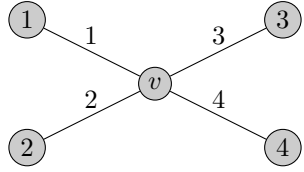

\begin{Ex}\label{ex4}\rm
Let $D$ be the distance matrix
\[
D=\begin{pmatrix}
0 & 3 & 4 & 5\\ 
3 & 0 & 5 & 6\\
4 & 5 & 0 & 7 \\
5 & 6 & 6 & 0 
\end{pmatrix}\]
From its compaction vector
\[
{\mathbf a}_D=\left( 1, 2,3,4\right)
\]
we get
\[
D({\mathbf a}_D)=D-E_1 -2E_2-3E_3-4E_4=
\begin{pmatrix}
0 & 0 & 0 &0\\
0 & 0 & 0 &0\\
0 & 0 & 0 &0\\
0 & 0 & 0 &0
\end{pmatrix}.
\]
Then, by Proposition \ref{propstar}, $D$ has a realization by the tree in Figure \ref{albex4}.

\end{Ex}

\section{Realizations by cycles}

In this section we give a useful result for matrices which are realizable by cycles. While in other research papers the order of the vertices is given a priori, or the vertices are eventually re-labelled  after some algorithm (see, for example, Definition 4.2 in \cite{BR}), the following theorem gives a more general characterization of such matrices also if the order of vertices is not known.

Let us denote by $(i_1, i_2, \dots, i_n)$ the permutation on $[n]$ such that
\[
\begin{array}{cccc}
1 & 2 & \cdots & n\\
\downarrow & \downarrow &  & \downarrow\\
i_1 & i_2 & \cdots & i_n
\end{array}
\]

We recall that a permutation  $\pi$ of $[n]$ is {\bf real} if it has only one term in the cycle-notation. For example $\pi_1=(2,3,4,5,1)$ is real, while $\pi_2=(2, 1, 4, 5,3)$ is not since it can be written as  $(2,1)(4,5,3)$.

\begin{thm} \label{n-agono}
A distance matrix $D$, of order $n$, has a realization by an $n-$cycle $C$ if and only if there exists a real permutation $\pi$ of $[n]$, such that
\begin{equation}\label{eqngonreal}
 d_{i \pi ^s(i)} = \min \left\{\sum_{t=0}^{s-1}d_{\pi^t(i) \pi^{t+1} (i)},  \sum_{t=s}^{n-1} d_{\pi^t(i) \pi^{t+1}(i)} \right\}
\end{equation}
for all $i=1,...,n$ and  $s=1,...,n-1$.
\end{thm}

\begin{proof}
If $D$ has a realization by an $n-$cycle $C$, let $i_1, i_2, \dots, i_n$, with $i_1=1$, be the labels of the nodes taken in clockwise order. Then $\pi=(i_2, \dots, i_n,1)$ is a real permutation and  since, on an $n-$cycle, the distance between nodes $i$ and $\pi ^s(i)$  is the minimum of the two paths from $i$ to $\pi ^s(i)$, equations (\ref{eqngonreal}) are satisfied, for all $i=1,...,n$ and  $s=1,...,n-1$.

Vice versa, suppose there exists a real permutation $\pi$ for which $D$ satisfies equation  (\ref{eqngonreal}), for all $i=1,...,n$ and  $s=1,...,n-1$. Then let $C$ be the graph of vertices  $V=\{1, 2, \dots, n\}$ and edges
\[
E=\{ (\pi^i(1),\pi^{i+1}(1)), \, i=0, \dots,n-1\}.
\]
It is easy to check that $C$ is an $n-$cycle.
Finally, putting the weight $d_{\pi^i(1)\pi^{i+1}(1)}$ on the edge $(\pi^i(1),\pi^{i+1}(1))$, for all $i=0, \dots,n-1$, we get a weighted $n-$cycle that, by equation (\ref{eqngonreal}), is a realization of $D$.
\end{proof}

\begin{Rmk}\label{rmkperm}\rm
According to the fact that, on an $n-$cycle, the distance between nodes $i$ and $j$  is the minimum of the two paths from $i$ to $j$, the reader could be surprised of formula (\ref{eqngonreal}) instead of the following formula
\begin{equation}\label{eqngon}
d_{ij}=\min \{d_{ii+1}+d_{i+1i+2}+\cdots + d_{j-1,j},d_{ii-1}+d_{i-1i-2}+\cdots + d_{j+1,j} \}
\end{equation}
where the indices are taken modulo $n$. Notice that (\ref{eqngon}) expresses the condition of  a path for an $n-$cycle where the indices are labeled in clockwise order.

However in general, it is not true that if a distance matrix $D$, of order $n\times n$, has a realization by an $n-$cycle $C$, then the adjacent vertices of $C$ correspond to adjacent rows (or columns) of $D$.

For example, if we consider the distance matrix
\[
D=\begin{pmatrix}
0&3&4&1&3&2\\
3&0&2&4&3&1\\
4&2&0&3&1&3\\
1&4&3&0&2&3\\
3&3&1&2&0&4\\
2&1&3&3&4&0
\end{pmatrix}
\]
we can notice that (\ref{eqngon}) does not work for some choice of $i$ and $j$. For example

\begin{equation}\label{casocattivo}
\begin{split}
3=d_{15}&\not=\min\{ d_{12}+d_{23}+d_{34}+d_{45},  d_{16}+d_{65}  \}\\
&=\min\{ 3+2+3+2,  2+4 \}\\
&=\min\{ 10,  6 \}=6
\end{split}
\end{equation}
However, $D$ has a realization by the $6-$cycle in Figure \ref{esa2}. Here we can notice that the vertices are ordered following the permutation $\pi=(4,5,3,2,6,1)$. Since $\pi^2(1)=5$, we can compute again (\ref{casocattivo}) using  (\ref{eqngonreal}), obtaining

{\small{\begin{equation*}
\begin{split}
3&=d_{15}=d_{1\pi^2(1)}=\\
&=\min\{ d_{1\pi(1)}+d_{\pi(1)\pi^2(1)},  d_{\pi^2(1)\pi^3(1)}+d_{\pi^3(1)\pi^4(1)}+d_{\pi^4(1)\pi^5(1)} + d_{\pi^5(1)1} \}\\
&=\min\{ d_{14}+d_{45},  d_{53}+d_{32}+d_{26} + d_{61} \}\\
&=\min\{ 1+2, 1+2+1+2 \}\\
&=\min\{ 3,  6 \}=3.
\end{split}
\end{equation*}}}

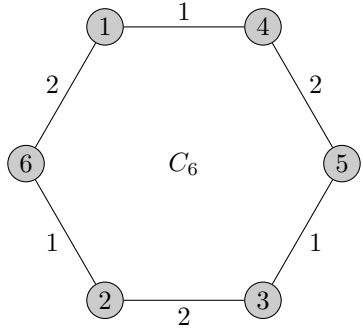
\begin{figure}
\centering
\begin{tikzpicture}
  [scale=0.7,auto=left]
  \node[circle,draw=black!100,fill=black!20,inner sep=2pt] (n1) at (0:3cm) {$5$};
  \node[circle,draw=black!100,fill=black!20,inner sep=2pt]  (n2) at (60:3cm) {$4$};
  \node[circle,draw=black!100,fill=black!20,inner sep=2pt] (n3) at (120:3cm) {$1$};
    \node[circle,draw=black!100,fill=black!20,inner sep=2pt] (n4) at (180:3cm) {$6$};
  \node[circle,draw=black!100,fill=black!20,inner sep=2pt]  (n5) at (240:3cm) {$2$}; 
  \node[circle,draw=black!100,fill=black!20,inner sep=2pt]  (n6) at (300:3cm) {$3$};

 \foreach \from/ \to in {n1/n2,n2/n3,n3/n4,n4/n5,n5/n6,n6/n1}
    \draw (\from) ->(\to);

\draw (0,2.9) node {$1$};
\draw (2.5,1.5) node {$2$};
\draw (2.5,-1.5) node {$1$};
\draw (0,-2.9) node {$2$};
\draw (-2.5,1.5) node {$2$};
\draw (-2.5,-1.5) node {$1$};
\draw (0,0) node {$C_6$};

\end{tikzpicture}
\caption{Realization of the matrix in Remark \ref{rmkperm}.}\label{esa2}
\end{figure}
\end{Rmk}

\begin{Rmk}\label{remarknodes} \rm The cases $s=1, s=n-1$ can be omitted since we would obtain two equal terms given by the unique monomial $d_{i\pi(i)}$ or $d_{i\pi^{n-1}(i)}$ and then the equation is trivially satisfied.
\end{Rmk}

Given a metric $D$ on $X$ with $|X|\geq 4$, we say that $D$ is {\it cyclelike} if there is a minimal realization for $D$ that is a cycle. This type of metric was discussed in, e.g., \cite{BR} ,\cite{ISPZ}, \cite{SP87}. In \cite{BR}, the authors characterize cyclelike metrics according to equation (\ref{eqngon}), after performing an algorithm that re-enumerate rows and columns of the metric. 

We want to mention a useful result from \cite{ISPZ}.

\begin{thm}[\cite{ISPZ}, Theorem 4.4]\label{cyclelike}
Suppose $D$ is a cyclelike metric on a finite set $X$ and a cycle $C$ is a minimal realization of  $D$ with $V(C)=X=\{v_1, v_2, \dots, v_m\}$, $m\geq 4$ and $E(C)=\{\{v_i,v_{i+1}\}\, : \, 1\leq i \leq m\}$, where the indices are taken modulo $m$. Then, $C$ is an optimal realization of $D$ if and only if
\begin{equation}\label{eqoptcycle}
D(v_{i-1},v_i)+D(v_i,v_{i+1})=D(v_{i-1},v_{i+1}) 
\end{equation}
holds for all $i$. In this case, $C$ is the unique optimal realization of $D$.
\end{thm}

It is not true, in general, that a metric $D$ satisfying the conditions of Theorem \ref{n-agono}, must satisfy also equations  (\ref{eqoptcycle}), where the order of the indices is given by the permutation $\pi$. However, this becomes true when we add an extra condition on $D$, thus we have the following result.

\begin{Prop}\label{optimal}
Let $D$ be a distance matrix such that ${\bf a}_D$ is the null vector. If $D$ satisfies the conditions of Theorem \ref{n-agono}, then the realization of $D$ by a cycle $C$ is optimal.
\end{Prop}

\begin{proof}
For simplicity of notation, we assume that $D$ satisfies equations (\ref{eqngonreal}) with respect to the permutation $\pi=(2,3,\dots, n-1,n,1)$. 

Suppose that the realization $C$ is not optimal, then, according to Theorem \ref{cyclelike} there are vertices $v_{i-1}, v_i, v_{i+1}$ such that
\begin{equation}\label{contrast}
D(v_{i-1},v_{i+1}) < D(v_{i-1},v_i)+D(v_i,v_{i+1}).
\end{equation}
Since $D$ satisfies the conditions of Theorem \ref{n-agono} we must have
\[
D(v_{i-1},v_{i+1})=D(v_{i+1},v_{i+2})+D(v_{i+2},v_{i+3})+\cdots+D(v_{i-2},v_{i-1}),
\]
where the indices are taken modulo $n$.
Let us add an edge between nodes $v_{i-1}$ and $v_{i+1}$ of weight $W:=D(v_{i+1},v_{i+2})+D(v_{i+2},v_{i+3})+\cdots+D(v_{i-2},v_{i-1})$. We can then apply an elementary cycle reduction (see \cite{HY}) obtaining a new node $w$ and the following set of weighted edges
\begin{itemize}
    \item[$\bullet$] the edge $(v_{i-1},w)$ of weight $\frac{1}{2}\left(D(v_{i-1},v_{i})+W-D(v_{i},v_{i+1}\right))$
    \item[$\bullet$] the edge $(v_i,w)$ of weight $\frac{1}{2}\left(D(v_{i-1},v_{i})+D(v_{i},v_{i+1})-W\right)$
    \item[$\bullet$] the edge $(v_{i+1},w)$ of weight $\frac{1}{2}\left(D(v_{i+1},v_{i})+W-D(v_{i-1},v_{i}\right))$
\end{itemize}

By (\ref{contrast}), the edge $(v_i,w)$ has  strictly positive weight and this number is exactly the $i-$th entry of the compaction vector ${\bf a}_D$, which is a contradiction.
\end{proof}

The application of Theorem \ref{n-agono} requires, in principle, to check all possible $n!$ permutations, but this  is intractable also for $n$ of  moderate size. However, the following algorithm permits to find an expected permutation $\pi$ in $O(n^2)$ time. 
\vskip0.2cm
INPUT: a distance matrix $D$ of order $n$ 
\begin{itemize}
    \item[(1)] set $\pi_n=1$ and $L=\{1\}$
    \item[(2)] choose the two minimal entries $D(1,i)$ and $D(1, j)$ in the first row (if more, choose two of them);
    \item[(3)] set $\pi_2=i$, $\pi_{n-1}=j$ and $L=\{1, i,j\}$
    \item[(4)] for $s$ from 2 to $n-2$ do
    \begin{itemize}
        \item[(4.1)] choose minimal entry $D(\pi_s,k)$ in $\pi_s$-th row with $k\notin L$ (if more, choose one of them);
        \item[(4.2)] set $\pi_{s+1}=k$ and $L:=L\cup \{k\}$
    \end{itemize}
\item[(5)] if $D(\pi_{n-1},j)$ is one of the two minimal entries in $\pi_{n-1}$-th row, then $\pi$ is the expected permutation
\end{itemize}

\section{Algorithm of realization of a matrix by a graph of genus 1} \label{AL}

The idea of the algorithm is to iterate the processes of compaction and reduction, obtaining at every step $i$ a new matrix $D(i)$. The algorithm will stop when, at a certain step $t$, we will be in one of the following cases
\begin{itemize}
\item[(1)] the matrix $D(t)$  has order 2;
\item[(2)] the compaction matrix $D({\mathbf a}_{D(t)})$ is the null matrix;
\item[(3)] the compaction vector  ${\mathbf a}_{D(t)}$ is the null vector and neither cases (1) or (2) are verified.
\end{itemize}

If we have kept track of all compaction vectors ${\mathbf a} _ {D(i)}$ and of all the rows and columns eliminated in the matrices $D ({\mathbf a} _ {D(i)})$, by going backwards, we can construct the graph realizing the starting matrix $D$, from the  graph $G(t)$ which realizes the last matrix $D(t)$, adding  at  every step $t-i$, with $1 \leq i \leq t$, new distinct nodes with edges adjacent to the graph $G_ {t-i + 1}$.

From Propositions \ref{proprank2} and \ref{propstar} we know that in cases (1) and (2), the matrix $D_t$ is realizable by a tree and hence, as we have just said, we will get that $D$ is realizable by a tree. Finally, if  we are in case (3) and we verify, by  Theorem \ref{n-agono},  that $D(t)$ is realizable by an $m-$cycle, with $m \leq n$ then the  matrix $D$ can be realized by a graph of genus 1. 

We divide the algorithm into two parts: a first part of analysis, which determines if or not a distance matrix can be realized by a graph of genus 1 and then, a part of reconstruction which, using data from the analysis part, constructs the desired graph of genus 1.
\vspace{0.9cm}

{\bf {\large Algorithm}}
\vspace{0.5cm}

{\bf INPUT:} a distance matrix  $D$ of order $n$.
\vspace{0.5cm}

{\underline {\bf Analysis}}:
\vskip0.3cm
\begin{itemize}
\item[{\bf Step 0:}] Let $t=0$, $D(t)=D$, $\Theta=\{1, 2, \dots, n\}$ and $\rho=n$.
\item[{\bf Step 1:}] Compute the compaction vector ${\mathbf a}_{D(t)}$ of $D(t)$.

 \noindent If ${\mathbf a}_{D(t)}$ is the null vector go to  {\bf Step 4}.
\item[{\bf Step 2:}] Compute the compaction matrix $D({\mathbf a}_{D(t)})$ of $D(t)$.
If  $D({\mathbf a}_D(t))$ is the zero matrix, then, by Proposition \ref{propstar}, $D(t)$ is realizable by a tree $G(V,E)$ and go to {\bf Step 5}.
\item[{\bf Step 3:}] Let
{\small{\[
S(t)=\Big\{ \big\{i_{1,1}(t),i_{1,2}(t), \dots, i_{1,s_1(t)}(t)\big\},\dots, \big\{i_{\theta(t),1}(t),i_{\theta(t),2}(t), \dots, i_{\theta(t),s_\theta(t)}(t)\big\}  \Big\}
\]}}
\noindent be the collection of all distinct (ordered) subset of $\Theta$ of indices of equal rows of $D({\mathbf a}_{D(t)})$, with $|S(t)|=\theta(t)$, 
and let 
\[
S'(t)=\big\{j_{1}(t),j_{2}(t), \dots, j_{\sigma(t)}(t)\big\}
\]
 be the set of (ordered) indices of $\Theta$ not in any subset of $S(t)$, with $|S'(t)|=\sigma(t)$.
 \vskip0.2cm

For $k$ from 1 to $\theta(t)$ do
\begin{enumerate}[leftmargin=2.3cm]
\item[{\bf Step 3.1:}] Remove from $D(t)$ all rows and columns indexed by $i_{{k,2}}(t)$, $i_{{k,3}}(t)$, $\dots  i_{{k,s_k(t)}}(t)$
\item[{\bf Step 3.2:}] Relabel $i_{k,1}(t)$-th row and column of $D(t)$ by $\rho+k$
\item[{\bf Step 3.3:}] Set 
\[
\Theta:=\Big(\Theta\cup\{\rho+k\}\Big)\setminus \big\{ i_{k,1}(t),i_{k,2}(t),\dots, i_{k,s_k}(t)\big\}.
\]
\end{enumerate}
\vskip0.2cm

For $k$ from 1 to $\sigma(t)$ do
\begin{enumerate}[leftmargin=2.3cm]
\item[{\bf Step 3.4:}] Relabel $j_{k}(t)$-th row and column of $D(t)$ by $\rho+\theta(t)+k$
\item[{\bf Step 3.5:}] Set 
\[
\Theta:=\Big(\Theta\cup\{\rho+\theta(t)+k\}\Big)\setminus \big\{ j_{k}(t)\big\}.
\]
\end{enumerate}

We get the reduction matrix $\hat{D}({\mathbf a}_{D(t)})$ with a new labelling of rows and columns.

\noindent If  $\hat{D}({\mathbf a}_{D(t)})$ has order 2 then, by Proposition \ref{proprank2},  $D(t)$ is realizable by a tree $G(V,E)$ and go to {\bf Step 5}.

\noindent Otherwise set
\[
\begin{array}{l}
\rho:=\rho+\theta(t)+\sigma(t),\\
\\
t:=t+1,\\
\\
D(t):=\hat{D}({\mathbf a}_{D(t-1)})
\end{array}
\]
and return to {\bf Step 1}.
\item[{\bf Step 4:}] Check if $D(t)$ is realizable by an $m-$cycle $G(V,E)$ with $m=|\Theta|$. If so, then $D$ can be realized by a graph of genus 1 and go to {\bf Step 5}. 

\noindent Otherwise, $D$ has not a realization by a tree or a graph of genus 1 then EXIT.
\end{itemize}
\vspace{0.5cm}

{\underline {\bf Reconstruction}}:
\vskip0.3cm
\begin{itemize}
\item[{\bf Step 5}] For  $\tau$ from $t-1$ to 0 do

\noindent Set $\rho:=\rho-\theta(\tau)-\sigma(\tau)$
\begin{enumerate}[leftmargin=1cm]
\item[{\bf Step 5.1:}]  for $\kappa$ from 1 to $\sigma(\tau)$ do
\vskip0.1cm
If $({\mathbf a}_{D(\tau)})_{j_{\kappa}(\tau)}\not=0$ then
\vskip0.1cm
\hspace{0.1cm} add node: $V:=V\cup \big\{j_{\kappa}(\tau) \big\}$
\vskip0.1cm
\hspace{0.1cm} 
add weighted edge: $E:=E\cup \big\{(\rho+\theta(\tau)+\kappa,j_{\kappa}(\tau),({\mathbf a}_{D(\tau)})_{j_{\kappa}(\tau)}) \big\}$
\vskip0.1cm
Else 
\vskip0.1cm
\hspace{0.1cm} replace node: $V:=\big(V\setminus\{\rho+\theta(\tau)+k\}\big) \cup \big\{j_{\kappa}(\tau) \big\}$
\vskip0.1cm

\item[{\bf Step 5.2:}]  for $\kappa$ from 1 to $\theta(\tau)$ do
\vskip0.1cm
Add nodes: $
V:=V\cup \big\{i_{\kappa,1}(\tau) ,i_{\kappa,2}(\tau),\dots,  i_{\kappa,s_\kappa(\tau)}(\tau) \big\}$
\vskip0.1cm
Add weighted edges:
{\small{\[
E:=E\cup \big\{(\rho+\kappa,i_{\kappa,1}(\tau),({\mathbf a}_{D(\tau)})_{i_{\kappa,1}(\tau)}) ,\dots,  (\rho+\kappa,i_{\kappa,s_\kappa}(\tau),({\mathbf a}_{D(\tau)})_{i_{\kappa,s_\kappa(\tau)}(\tau)}) \big\}
\]}}
\end{enumerate}
\end{itemize}

\vskip0.2cm
 {\bf OUTPUT:} the graph $G$ realizing $D$. 
\vskip0.5cm

\begin{Rmk}\rm
It is easy to observe that, in the algorithm, the equal rows in a compaction matrix will correspond to nodes which are adjacent to the same interior node.
\end{Rmk}

\begin{Rmk}\rm
The algorithm has been implemented in \texttt{Maple}${}^\mathrm{TM}$. Here, the stored data of compaction vectors and reduction matrices are used to build the weighted adjacency matrix of the graph which realizes the distance matrix $D$.
The algorithm can be found in the ancillary file \texttt{algorithm.mw}.
\end{Rmk}

The study of recursive compaction vectors is equivalent to removing all pendant trees in the optimal realization. In such a procedure, the number of operations grows in a time complexity of $O(n^4)$, where $n$ is the order of the distance matrix $D$. As a matter of fact, computing an entry of the compaction vector is done in a time complexity of $O(n^2)$. This must be done for the $n$ entries of the compaction vector. Moreover, we need to compute the compaction vector for the successive steps of the algorithm. The number of steps is at most $n$.
As for storage requirement, at the end of each step we store compaction vectors and reduction matrices, this is done in time complexity of $O(n^3)$.
\vskip0.2cm

We conclude this section with a result about optimality of this realization.

\begin{Prop}
If the algorithm outputs a graph of genus 1 realizing $D$, then this realization is optimal.
\end{Prop}

\begin{proof}
In the analysis part, the algorithm checks if the matrix $D(t)$ satisfies equations (\ref{eqngonreal}) at Step 4, which is reached if the compaction vector of $D(t)$ is the null vector. Hence, by Proposition \ref{optimal}, if $D(t)$ has a realization by a cycle $C$, this realization is optimal.

Successively, the reconstruction part of the algorithm reconstructs the graph $G$ realizing $D$, starting from $C$ and adding, recursively,  pendant edges. Hence, by the following theorem, the realization of $D$  is optimal.
\end{proof}

\begin{thm}[\cite{HY}, Theorem 5]
If $0\leq a \leq a_i$ and if $G_i(a)$ is an optimal realization of $D_i(a)$, then $G$ obtained from $G_i(a)$ adding the vertex $v'_i$ to $G_i(a)$ and the edge $(v'_i,v_i)$ of weight $a$, is an optimal realization of $D$.
\end{thm}

\section{Examples}
The following   two examples  show how the algorithm works.
 
 \begin{Ex}\label{exalgex1}\rm
Consider the following matrix
\[
D=
\bordermatrix{ & {\mathbf 1}& {\mathbf 2}& {\mathbf 3} & {\mathbf 4} & {\mathbf 5} & {\mathbf 6}\cr
{\mathbf 1} &0 & 2 & 5 & 6 & 7 & 4 \cr
{\mathbf 2} & 2 & 0 & 5 & 6 & 7 & 4 \cr
{\mathbf 3}&5 & 5 & 0 & 4 & 5 &3\cr
{\mathbf 4}&6 & 6 & 4 & 0 & 3 & 2\cr
{\mathbf 5} &7 & 7 & 5 & 3 & 0 & 3\cr
{\mathbf 6}& 4 & 4 & 3 & 2 & 3 &0
}
\]
We set $t=0$, $D(0)=D$, $\Theta=\{1,2,3,4,5,6\}$ and $\rho=6$.
The compaction vector of $D(0)$ is ${\mathbf a}_{D(0)}=(1,1,\frac{3}{2},1,2,0)$. Since ${\mathbf a}_{D(0)}$ is not the zero vector we pass to {\bf Step 2} computing the compaction matrix
\[
D({\mathbf a}_{D(0)})=
\bordermatrix{ & {\mathbf 1} & {\mathbf 2} & {\mathbf 3} & {\mathbf 4} & {\mathbf 5} & {\mathbf 6}\cr
{\mathbf 1}& 0 & 0 & \frac{5}{2} & 4 & 4 & 3 \cr
{\mathbf 2} & 0 & 0 & \frac{5}{2} & 4 & 4 & 3 \cr
{\mathbf 3} & \frac{5}{2} & \frac{5}{2} & 0 & \frac{3}{2} & \frac{3}{2} &\frac{3}{2}\cr
{\mathbf 4} & 4 & 4 & \frac{3}{2} & 0 & 0 & 1\cr
{\mathbf 5} & 4 & 4 & \frac{3}{2} & 0 & 0 & 1\cr
{\mathbf 6} & 3 & 3 & \frac{3}{2} & 1 & 1 &0
}
\]
Since $D({\mathbf a}_{D(0)})$ is not the zero matrix we pass to {\bf Step 3}. Notice that rows 1 and 4 are respectively equal to rows 2 and 5. Hence $\theta(0)=2$, $\sigma(0)=2$ and
\[
S(0)=\Big\{ \big\{1,2 \big\},\big\{4,5\big\}  \Big\} \quad S'(0)=\big\{3,6 \big\}
\]

Hence, we remove rows (and columns) labeled by 2 and 5 and relabel 
\begin{itemize}
\item[$\bullet$] rows (and columns) 1 and 4 respectively by 7($=\rho+1$) and 8($=\rho+2=\rho+\theta(0)$),
\item[$\bullet$] rows (and columns) 3 and 6 respectively by 9($=\rho+\theta(0) +1$) and 10($=\rho+\theta(0)+2=\rho+\theta(0)+\sigma(0)$).
\end{itemize}
We get
\[
\Theta=\{1,2,3,4,5,6\}\cup\{7,8,9,10\}\setminus\{1,2,3,4,5,6\}=\{7,8,9,10\}
\]
and
\[
\hat{D}({\mathbf a}_{D(0)})=
\bordermatrix{ &  {\mathbf 7} & {\mathbf 9} & {\mathbf 8} & {\mathbf 10} \cr
{\mathbf 7} & 0 & \frac{5}{2} & 4 & 3 \cr
{\mathbf 9} &  \frac{5}{2} & 0 & \frac{3}{2}  &\frac{3}{2}\cr
{\mathbf 8} & 4 & \frac{3}{2}  & 0 & 1\cr
{\mathbf 10}&  3 & \frac{3}{2} & 1 &0
}
\]
Since $\hat{D}({\mathbf a}_{D(0)})$ has not order 2, we come back to {\bf Step 1} after setting $t=1$, $\rho=10$ and $D(1)=\hat{D}({\mathbf a}_{D(0)})$.

The compaction vector of $D(1)$ is ${\mathbf a}_{D(1)}=(2,0,\frac{1}{2},0)$ and again we pass to {\bf Step 2} computing the compaction matrix 
\[
D({\mathbf a}_{D(1)})=
\bordermatrix{ & {\mathbf7} & {\mathbf 9} & {\mathbf 8} & {\mathbf 10} \cr
{\mathbf 7} &0 & \frac{1}{2} & \frac{3}{2} & 1\cr
{\mathbf 9} &\frac{1}{2} & 0 & 1 & \frac{3}{2}\cr
{\mathbf 8} &\frac{3}{2} & 1 & 0 & \frac{1}{2}\cr
{\mathbf 10} &1 & \frac{3}{2} & \frac{1}{2} & 0
}
\]
This matrix has not equal rows so $\theta(1)=0$ and $\sigma(1)=4$ with $S'(1)=\{7,9,8,10\}$. After {\bf Step 3.4} and {\bf Step 3.5} we get $\Theta=\{ 11,12,13,14\}$ and
\[
\hat{D}({\mathbf a}_{D(1)})=
\bordermatrix{ & {\mathbf 11} & {\mathbf 12} & {\mathbf 13} & {\mathbf 14} \cr
{\mathbf 11} &0 & \frac{1}{2} & \frac{3}{2} & 1\cr
{\mathbf 12} &\frac{1}{2} & 0 & 1 & \frac{3}{2}\cr
{\mathbf 13} &\frac{3}{2} & 1 & 0 & \frac{1}{2}\cr
{\mathbf 14} &1 & \frac{3}{2} & \frac{1}{2} & 0
}
\]
Since it has not order 2, we set $t=2$, $\rho=14$ and $D(2)=\hat{D}({\mathbf a}_{D(1)})$.
 
Computing the compaction vector of $D(2)$ we get the null vector. Hence, in the analysis part of the algorithm  we move to {\bf Step 4}.
 
Notice that $D(2)$ satisfies Formula (\ref{eqngonreal}) of Theorem \ref{n-agono} for $\pi=(2,3,4,1)$ hence it is realizable by the $n-$cycle $G$ in Figure \ref{exalg1}(a) and, moreover, $D$ will be realizable by a graph of genus 1.

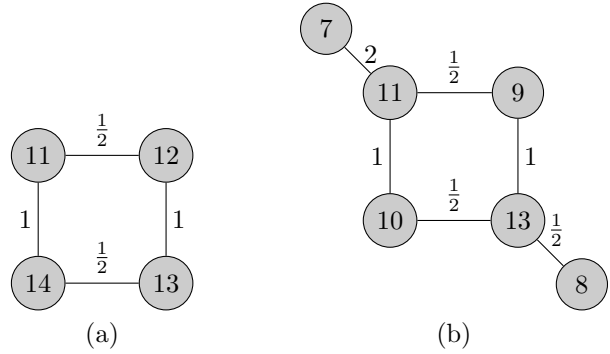
\begin{figure}
\centering
\begin{tabular}{ccc}
\begin{tikzpicture}
  [scale=0.85,auto=left]
  \node[circle,draw=black!100,fill=black!20,inner sep=3pt] (n1) at (0,2) {$11$};
  \node[circle,draw=black!100,fill=black!20,inner sep=3pt]  (n2) at (2,2) {$12$};
  \node[circle,draw=black!100,fill=black!20,inner sep=3pt]  (n3) at (2,0) {$13$};
  \node[circle,draw=black!100,fill=black!20,inner sep=3pt]  (n4) at (0,0) {$14$};
 
 \foreach \from/ \to in {n1/n2, n2/n3,n3/n4, n4/n1}
    \draw (\from) ->(\to);

\draw (1,2.4) node[scale=1] {$\frac{1}{2}$};
\draw (1,0.4) node[scale=1] {$\frac{1}{2}$};
\draw (-0.2,1) node[scale=1] {$1$};
\draw (2.2,1) node[scale=1] {$1$};

\end{tikzpicture} & $\qquad$ & 
\begin{tikzpicture}
  [scale=0.85,auto=left]
  \node[circle,draw=black!100,fill=black!20,inner sep=3pt] (n1) at (0,2) {$11$};
  \node[circle,draw=black!100,fill=black!20,inner sep=4pt]  (n2) at (2,2) {$9$};
  \node[circle,draw=black!100,fill=black!20,inner sep=3pt]  (n3) at (2,0) {$13$};
  \node[circle,draw=black!100,fill=black!20,inner sep=3pt]  (n4) at (0,0) {$10$};
  \node[circle,draw=black!100,fill=black!20,inner sep=4pt]  (n5) at (-1,3) {$7$};
  \node[circle,draw=black!100,fill=black!20,inner sep=4pt]  (n6) at (3,-1) {$8$};
 
 \foreach \from/ \to in {n1/n2, n2/n3,n3/n4, n4/n1, n1/n5, n3/n6}
    \draw (\from) ->(\to);

\draw (1,2.4) node[scale=1] {$\frac{1}{2}$};
\draw (1,0.4) node[scale=1] {$\frac{1}{2}$};
\draw (-0.2,1) node[scale=1] {$1$};
\draw (2.2,1) node[scale=1] {$1$};
\draw (-0.3,2.6) node[scale=1] {$2$};
\draw (2.6,-0.15) node[scale=1] {$\frac{1}{2}$};

\end{tikzpicture}\\
(a) && (b)
\end{tabular}
\caption{Intermediate steps for the realization of $D$ of Example \ref{exalgex1}.}\label{exalg1}
\end{figure}

We start now the reconstruction part of the algorithm to determine such a graph.
Since for $\tau=1$($=t-1$) one has $\theta(1)=0$ and $\sigma(1)=4$ we set $\rho=10$ and we perform only {\bf Step 5.1}. More in details, one has that
\begin{itemize}
\item[$\bullet$] since $({\mathbf a}_{D(1)})_{\mathbf 7}=2$  we add node $\{ 7 \}$ and edge $(11,7,2)$.
\item[$\bullet$] since $({\mathbf a}_{D(1)})_{\mathbf 9}=0$  we replace node  $\{ 12 \}$ by node $\{ 9 \}$.
\item[$\bullet$] since $({\mathbf a}_{D(1)})_{\mathbf 8}=\frac12$  we add node $\{ 8 \}$ and edge $(13,8,\frac12)$.
\item[$\bullet$] since $({\mathbf a}_{D(1)})_{\mathbf 10}=0$  we replace node  $\{ 14 \}$ by node $\{ 10 \}$.
\end{itemize}
Thus, we get the graph in Figure \ref{exalg1}(b).

Now we perform one more time {\bf Step 5} for $\tau=0$ (hence $\rho=10-\theta(0)-\sigma(0)=6$).
Performing {\bf Step 5.1} one has that
\begin{itemize}
\item[$\bullet$] since $({\mathbf a}_{D(0)})_{\mathbf 3}=\frac{3}{2}$  we add node $\{ 3 \}$ and edge $(9,3,\frac{3}{2})$.
\item[$\bullet$] since $({\mathbf a}_{D(0)})_{\mathbf 6}=0$  we replace node  $\{ 10 \}$ by node $\{ 6 \}$.
\end{itemize}
Performing {\bf Step 5.2} we add
\begin{itemize}
\item[$\bullet$] nodes $\{1\}$ and $\{2\}$ and  edges $(7,1,1)$ and $(7,2,1)$ 
\item[$\bullet$] nodes $\{4\}$ and $\{5\}$ and  edges $(8,4,1)$ and $(8,5,2)$ 
\end{itemize}
This  concludes the algorithm giving the graph of genus 1 in Figure \ref{exalg12}, which realizes $D$.

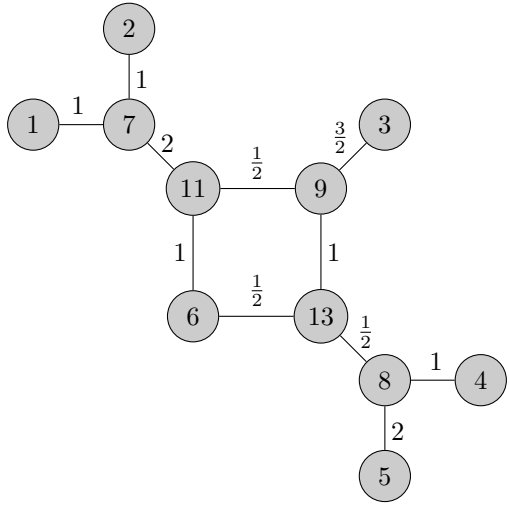
\begin{figure}
\centering
\begin{tikzpicture}
  [scale=0.85,auto=left]
  \node[circle,draw=black!100,fill=black!20,inner sep=3pt] (n1) at (0,2) {$11$};
  \node[circle,draw=black!100,fill=black!20,inner sep=4pt]  (n2) at (2,2) {$9$};
  \node[circle,draw=black!100,fill=black!20,inner sep=3pt]  (n3) at (2,0) {$13$};
  \node[circle,draw=black!100,fill=black!20,inner sep=4pt]  (n4) at (0,0) {$6$};
  \node[circle,draw=black!100,fill=black!20,inner sep=4pt]  (n5) at (-1,3) {$7$};
  \node[circle,draw=black!100,fill=black!20,inner sep=4pt]  (n6) at (3,-1) {$8$};
 \node[circle,draw=black!100,fill=black!20,inner sep=4pt]  (n7) at (-2.5,3) {$1$};
 \node[circle,draw=black!100,fill=black!20,inner sep=4pt]  (n8) at (-1,4.5) {$2$};
   \node[circle,draw=black!100,fill=black!20,inner sep=4pt]  (n9) at (3,3) {$3$};
 \node[circle,draw=black!100,fill=black!20,inner sep=4pt]  (n10) at (4.5,-1) {$4$};
  \node[circle,draw=black!100,fill=black!20,inner sep=4pt]  (n11) at (3,-2.5) {$5$};
  
 \foreach \from/ \to in {n1/n2, n2/n3,n3/n4, n4/n1, n1/n5, n3/n6, n5/n8, n5/n7, n2/n9, n6/n10, n6/n11}
    \draw (\from) ->(\to);

\draw (1,2.4) node[scale=1] {$\frac{1}{2}$};
\draw (1,0.4) node[scale=1] {$\frac{1}{2}$};
\draw (-0.2,1) node[scale=1] {$1$};
\draw (2.2,1) node[scale=1] {$1$};
\draw (-0.4,2.7) node[scale=1] {$2$};
\draw (2.7,-0.25) node[scale=1] {$\frac{1}{2}$};

\draw (-1.8,3.3) node[scale=1] {$1$};
\draw (-0.8,3.7) node[scale=1] {$1$};
\draw (2.3,2.8) node[scale=1] {$\frac{3}{2}$};
\draw (3.8,-0.7) node[scale=1] {$1$};
\draw (3.2,-1.8) node[scale=1] {$2$};    
    
\end{tikzpicture}
\caption{Graph of  genus  1 realizing matrix $D$ of Example \ref{exalgex1}.}\label{exalg12}
\end{figure}
\end{Ex}

\begin{Ex}\label{exalgex2}\rm
Consider the following matrix
\[
D=
\bordermatrix{ 
& {\mathbf 1}& {\mathbf 2}& {\mathbf 3} & {\mathbf 4} & {\mathbf 5} & {\mathbf 6} & {\mathbf 7} & {\mathbf 8} &{\mathbf 9}\cr
{\mathbf 1} &0 & 4 & 2 & 6 & 4 & 6 & 2 &2 & 5 \cr
{\mathbf 2} & 4 & 0 & 4 & 4 & 4 & 4 & 4 &4 & 4 \cr
{\mathbf 3}&2 & 4 & 0 & 6 & 4 &6 & 2 &2 & 5\cr
{\mathbf 4}&6 & 4 & 6 & 0 & 5 & 2 & 6 & 6 & 4\cr
{\mathbf 5} &4 & 4 & 4 & 5 & 0 & 5 &4 & 4 & 3\cr
{\mathbf 6}& 6 & 4 & 6 & 2 &5  &0 & 6 & 6 & 4\cr
{\mathbf 7}& 2 & 4 & 2 & 6 & 4 & 6 & 0 & 2 & 5\cr
{\mathbf 8}& 2 & 4 & 2 & 6 & 4 & 6 & 2 & 0 & 5\cr
{\mathbf 9}& 5& 4& 5& 4&3& 4& 5& 5& 0
}
\]
We set $t=0$, $D(0)=D$, $\Theta=\{1,2,3,4,5,6,7,8,9\}$ and $\rho=9$.
The compaction vector of $D(0)$ is ${\mathbf a}_{D(0)}=(1,1,1,1,1,1,1,1,1)$. Since ${\mathbf a}_{D(0)}$ is not the zero vector we pass to {\bf Step 2} computing the compaction matrix
\[
D({\mathbf a}_{D(0)})=
\bordermatrix{ 
& {\mathbf 1}& {\mathbf 2}& {\mathbf 3} & {\mathbf 4} & {\mathbf 5} & {\mathbf 6} & {\mathbf 7} & {\mathbf 8} &{\mathbf 9}\cr
 {\mathbf 1}&0 & 2 &0 & 4 & 2 &4 & 0 &0&3   \cr
 {\mathbf 2}&2 & 0 & 2 & 2 &2 & 2 & 2& 2 &2\cr
 {\mathbf 3}&0 & 2 &0 & 4 & 2 &4 & 0 &0&3   \cr
 {\mathbf 4}&4 & 2 & 4 & 0 & 3 & 0 & 4 & 4 & 2\cr
 {\mathbf 5}&2 & 2 & 2 & 3& 0 & 3 & 2 & 2 &1\cr
 {\mathbf 6}&4 & 2 & 4 & 0 & 3 & 0 & 4 & 4 & 2\cr
 {\mathbf 7}&0 & 2 &0 & 4 & 2 &4 & 0 &0 &3   \cr
 {\mathbf 8}&0 & 2 &0 & 4 & 2 &4 & 0 &0 &3 \cr
 {\mathbf 9}&3 & 2 &3 & 2 & 1 &2 & 3 &3 &0 \cr }
\]
Since $D({\mathbf a}_{D(0)})$ is not the zero matrix we pass to {\bf Step 3}. Here we have
\[
S(0)=\Big\{ \big\{1,3,7,8 \big\},\big\{4,6\big\}  \Big\} \quad S'(0)=\big\{2,5,9 \big\},
\]
\[
\theta(0)=2 \qquad \sigma(0)=3.
\]
Hence, we remove rows (and columns) labeled by 3, 7, 8 and 6 and relabel 
\begin{itemize}
\item[$\bullet$] rows (and columns) 1 and 4 respectively by 10($=\rho+1$) and 11($=\rho+2=\rho+\theta(0)$),
\item[$\bullet$] rows (and columns) 2,5 and 9 respectively by 12($=\rho+\theta(0) +1$), 13 and 14($=\rho+\theta(0)+3=\rho+\theta(0)+\sigma(0)$).
\end{itemize}
We get
\[
\Theta(t)=\{10,11,12,13,14\}
\]
and 
\[
\hat{D}({\mathbf a}_{D(0)})=\bordermatrix{ & {\mathbf 10} & {\mathbf  12} &  {\mathbf 11} & {\mathbf 13} &  {\mathbf 14}\cr
 {\mathbf 10} & 0 & 2 & 4 & 2& 3\cr
 {\mathbf 12} & 2 & 0 & 2 & 2& 2\cr
 {\mathbf 11}&  4 & 2 & 0 &3 & 2\cr
 {\mathbf 13} & 2 & 2 & 3& 0 & 1\cr
 {\mathbf 14}&3 & 2 & 2 & 1 &0
 }
\]
Since $\hat{D}({\mathbf a}_{D(0)})$ has not order 2, we come back to {\bf Step 1} after setting $t=1$, $\rho=14$ and $D(1)=\hat{D}({\mathbf a}_{D(0)})$.

The compaction vector of $D(1)$ is ${\mathbf a}_{D(1)}=(1,0,1,0,0)$ and again we pass to {\bf Step 2} computing the compaction matrix 
\[
D({\mathbf a}_{D(1)})=
\bordermatrix{ & {\mathbf 10} & {\mathbf 12} & {\mathbf 11} & {\mathbf 13} & {\mathbf 14} \cr
{\mathbf 10} &0 & 1 & 2 &1 & 2\cr
{\mathbf 12} &1& 0 & 1 &2 & 2\cr
{\mathbf 11} &2 & 1 & 0 & 2 & 1\cr
{\mathbf 13} &1 & 2&2 & 0 &1\cr
 {\mathbf 14}&2 & 2 & 1 & 1 &0
}
\]
This matrix has not equal rows so $\theta(1)=0$ and $\sigma(1)=5$ with 
\[
S'(1)=\{10,12,11,13,14\}.
\]
After {\bf Step 3.4} and {\bf Step 3.5} we get $\Theta=\{ 15,16,17,18,19\}$ and
\[
\hat{D}({\mathbf a}_{D(1)})=
\bordermatrix{ & {\mathbf 15} & {\mathbf 16} & {\mathbf 17} & {\mathbf 18} & {\mathbf 19} \cr
{\mathbf 15} &0 & 1 & 2 &1 & 2\cr
{\mathbf 16} &1& 0 & 1 &2 & 2\cr
{\mathbf 17} &2 & 1 & 0 & 2 & 1\cr
{\mathbf 18} &1 & 2&2 & 0 &1\cr
 {\mathbf 19}&2 & 2 & 1 & 1 &0
}
\]
Since it has not order 2, we set $t=2$, $\rho=19$ and $D(2)=\hat{D}({\mathbf a}_{D(1)})$.
 
Computing the compaction vector of $D(2)$ we get the null vector. Hence, in the analysis part of the algorithm  we move to {\bf Step 4}.
 
Notice that $D(2)$ satisfies Formula (\ref{eqngonreal}) of Theorem \ref{n-agono} for $\pi=(4,5,3,2,1)$ hence it is realizable by the $5$-cycle $G$  in Figure \ref{exalg2}(a) and, moreover, $D$ will be realizable by a graph of genus 1.

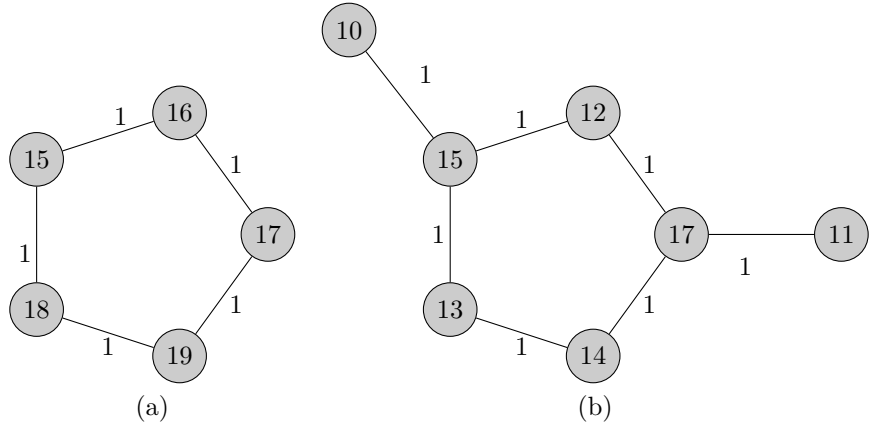
\begin{figure}
\centering
\begin{tabular}{cc}
\begin{tikzpicture}
  [scale=0.85,auto=left]
  \node[circle,draw=black!100,fill=black!20,inner sep=3pt] (n1) at (0:2cm) {$17$};
  \node[circle,draw=black!100,fill=black!20,inner sep=3pt]  (n2) at (72:2cm) {$16$};
  \node[circle,draw=black!100,fill=black!20,inner sep=3pt]  (n3) at (144:2cm) {$15$};
  \node[circle,draw=black!100,fill=black!20,inner sep=3pt]  (n4) at (216:2cm) {$18$};
  \node[circle,draw=black!100,fill=black!20,inner sep=3pt]  (n5) at (288:2cm) {$19$};
 
 \foreach \from/ \to in {n1/n2, n2/n3,n3/n4, n4/n5,n5/n1}
    \draw (\from) ->(\to);

\draw (1.5,1.1) node[scale=1] {$1$};
\draw (1.5,-1.1) node[scale=1] {$1$};
\draw (-0.3,1.85) node[scale=1] {$1$};
\draw (-0.5,-1.75) node[scale=1] {$1$};
\draw (-1.8,-0.3) node[scale=1] {$1$};    
    
\end{tikzpicture} &  
\begin{tikzpicture}
  [scale=0.85,auto=left]

  \node[circle,draw=black!100,fill=black!20,inner sep=3pt] (n1) at (0:2cm) {$17$};
  \node[circle,draw=black!100,fill=black!20,inner sep=3pt]  (n2) at (72:2cm) {$12$};
  \node[circle,draw=black!100,fill=black!20,inner sep=3pt]  (n3) at (144:2cm) {$15$};
  \node[circle,draw=black!100,fill=black!20,inner sep=3pt]  (n4) at (216:2cm) {$13$};
  \node[circle,draw=black!100,fill=black!20,inner sep=3pt]  (n5) at (288:2cm) {$14$};

 \node[circle,draw=black!100,fill=black!20,inner sep=3pt]  (n6) at (-3.2,3.2) {$10$};
  \node[circle,draw=black!100,fill=black!20,inner sep=3pt]  (n7) at (4.5,0) {$11$};

 \foreach \from/ \to in {n1/n2, n2/n3,n3/n4, n4/n5,n5/n1, n3/n6, n1/n7} 
    \draw (\from) ->(\to);

\draw (1.5,1.1) node[scale=1] {$1$};
\draw (1.5,-1.1) node[scale=1] {$1$};
\draw (-0.5,1.8) node[scale=1] {$1$};
\draw (-0.5,-1.75) node[scale=1] {$1$};
\draw (-1.8,0) node[scale=1] {$1$};  
    
\draw (-2,2.5) node[scale=1] {$1$};
\draw (3,-0.5) node[scale=1] {$1$};

\end{tikzpicture}\\
(a) & (b)
\end{tabular}
\caption{Intermediate steps for the realization of $D$ of Example \ref{exalgex2}.}\label{exalg2}
\end{figure}

We start now the reconstruction part of the algorithm to determine such a graph.
Since for $\tau=1$($=t-1$) one has $\theta(1)=0$ and $\sigma(1)=5$ we set $\rho=14$ and we perform only {\bf Step 5.1}. More in details, one has that
\begin{itemize}
\item[$\bullet$] since $({\mathbf a}_{D(1)})_{\mathbf 10}=1$  we add node $\{ 10 \}$ and edge $(15,10,1)$.
\item[$\bullet$] since $({\mathbf a}_{D(1)})_{\mathbf 12}=0$  we replace node  $\{ 16 \}$ by node $\{ 12 \}$.
\item[$\bullet$] since $({\mathbf a}_{D(1)})_{\mathbf 11}=1$  we add node $\{ 11 \}$ and edge $(17,11,1)$.
\item[$\bullet$] since $({\mathbf a}_{D(1)})_{\mathbf 13}=0$  we replace node  $\{ 18 \}$ by node $\{ 13 \}$.
\item[$\bullet$] since $({\mathbf a}_{D(1)})_{\mathbf 14}=0$  we replace node  $\{ 19 \}$ by node $\{ 14 \}$.
\end{itemize}
Thus, we get the graph in Figure \ref{exalg2}(b).

Now, we perform one more time {\bf Step 5} for $\tau=0$ (hence $\rho=14-\theta(0)-\sigma(0)=9$).
Performing {\bf Step 5.1} one has that
\begin{itemize}
\item[$\bullet$] since $({\mathbf a}_{D(0)})_{\mathbf 2}=1$  we add node $\{ 2 \}$ and edge $(12,2,1)$.
\item[$\bullet$] since $({\mathbf a}_{D(0)})_{\mathbf 5}=1$  we add node $\{ 5 \}$ and edge $(13,5,1)$.
\item[$\bullet$] since $({\mathbf a}_{D(0)})_{\mathbf 9}=1$  we add node $\{ 9 \}$ and edge $(14,9,1)$.

\end{itemize}
Performing {\bf Step 5.2} we add
\begin{itemize}
\item[$\bullet$] nodes $\{1\}$, $\{3\}$,$\{7\}$ and $\{8\}$  and  edges $(10,1,1)$, $(10,3,1)$ $(10,7,1)$ and $(10,8,1)$ 
\item[$\bullet$] nodes $\{4\}$ and $\{6\}$ and  edges $(11,4,1)$ and $(11,6,1)$ 
\end{itemize}
This concludes the algorithm giving the graph of genus 1 in Figure \ref{exalg22}, which realizes $D$.

\begin{figure}
\centering
\begin{tikzpicture}
  [scale=0.85,auto=left]
 
 \node[circle,draw=black!100,fill=black!20,inner sep=3pt] (n1) at (0:2cm) {$17$};
  \node[circle,draw=black!100,fill=black!20,inner sep=3pt]  (n2) at (72:2cm) {$12$};
  \node[circle,draw=black!100,fill=black!20,inner sep=3pt]  (n3) at (144:2cm) {$15$};
  \node[circle,draw=black!100,fill=black!20,inner sep=3pt]  (n4) at (216:2cm) {$13$};
  \node[circle,draw=black!100,fill=black!20,inner sep=3pt]  (n5) at (288:2cm) {$14$};

 \node[circle,draw=black!100,fill=black!20,inner sep=3pt]  (n6) at (-3.2,3.2) {$10$};
  \node[circle,draw=black!100,fill=black!20,inner sep=3pt]  (n7) at (4.5,0) {$11$};

 \node[circle,draw=black!100,fill=black!20,inner sep=4pt]  (n8) at (-4.9,1.7) {$1$};
 \node[circle,draw=black!100,fill=black!20,inner sep=4pt]  (n9) at (-5.2,4) {$3$};
 \node[circle,draw=black!100,fill=black!20,inner sep=4pt]  (n10) at (-3.4,5.5) {$7$};
  \node[circle,draw=black!100,fill=black!20,inner sep=4pt]  (n11) at (-1.3,4.8) {$8$};
  
    \node[circle,draw=black!100,fill=black!20,inner sep=4pt]  (n12) at (1.7,3.6) {$2$};
 \node[circle,draw=black!100,fill=black!20,inner sep=4pt]  (n13) at (6,1.5) {$4$};
  \node[circle,draw=black!100,fill=black!20,inner sep=4pt]  (n14) at (6,-1.5) {$6$};
  \node[circle,draw=black!100,fill=black!20,inner sep=4pt]  (n15) at (-3,-2.6) {$5$};
  \node[circle,draw=black!100,fill=black!20,inner sep=4pt]  (n16) at (1.1,-3.65) {$9$};
  
   \foreach \from/ \to in {n1/n2, n2/n3,n3/n4, n4/n5,n5/n1, n3/n6,n6/n8, n6/n9,n6/n10, n6/n11, n12/n2, n1/n7,n7/n13,n7/n14,n4/n15,n5/n16}
    \draw (\from) ->(\to);

 \draw (1.5,1.1) node[scale=1] {$1$};
\draw (1.5,-1.1) node[scale=1] {$1$};
\draw (-0.5,1.85) node[scale=1] {$1$};
\draw (-0.5,-1.8) node[scale=1] {$1$};
\draw (-1.8,0) node[scale=1] {$1$};  
    
\draw (-2.1,2.3) node[scale=1] {$1$};

\draw (-4.3,2.6) node[scale=1] {$1$};
\draw (-4.05,3.9) node[scale=1] {$1$};
\draw (-3.1,4.5) node[scale=1] {$1$};
\draw (-2.2,3.7) node[scale=1] {$1$};
\draw (0.9,2.9) node[scale=1] {$1$};
\draw (3.3,0.3) node[scale=1] {$1$};
\draw (5,0.9) node[scale=1] {$1$};
\draw (5,-0.9) node[scale=1] {$1$};

\draw (-2.4,-1.6) node[scale=1] {$1$};
\draw (1,-2.6) node[scale=1] {$1$};
  
\end{tikzpicture}
\caption{Graph of  genus  1 realizing matrix $D$ of Example \ref{exalgex2}.}\label{exalg22}
\end{figure}
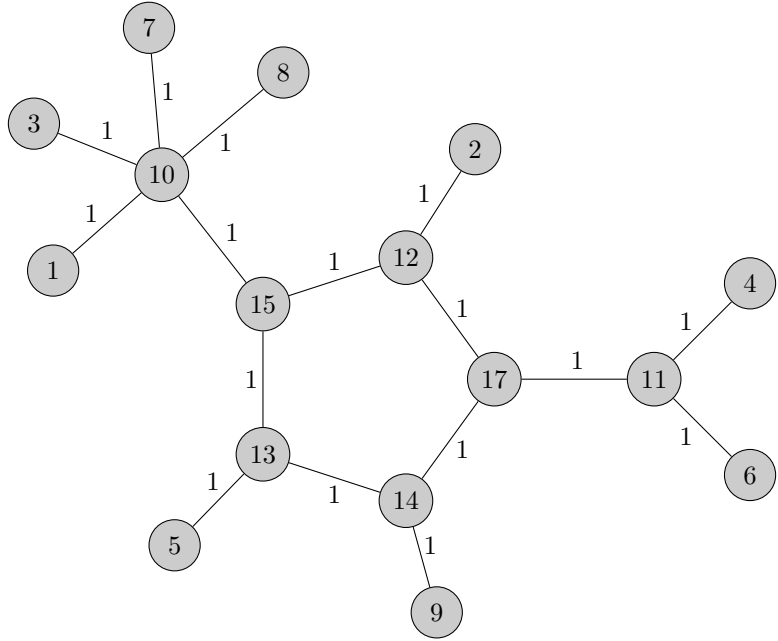
\end{Ex}

We end this section with some examples performed by our algorithm in Maple${}^\mathrm{TM}$. Here, starting from some weighted graphs of genus 1 we computed their distance matrices and we used them, as input, for the algorithm.
The Figures \ref{maple1} - \ref{maple4} show the original graph (on the left) and the output of the algorithm in Maple${}^\mathrm{TM}$(on the right).

\begin{figure}
\centering
\begin{tabular}{cc}
\begin{tikzpicture}
  [scale=0.5,auto=left]
   \node[circle,draw=black!100,fill=black!20,inner sep=7pt,scale=0.4] (n1) at (0:2.2cm) {};
  \node[circle,draw=black!100,fill=black!20,inner sep=7pt,scale=0.4]  (n2) at (60:2.2cm) {};
  \node[circle,draw=black!100,fill=black!20,inner sep=7pt,scale=0.4]  (n3) at (120:2.2cm) {};
  \node[circle,draw=black!100,fill=black!20,inner sep=7pt,scale=0.4]  (n4) at (180:2.2cm) {};
  \node[circle,draw=black!100,fill=black!20,inner sep=7pt,scale=0.4]  (n5) at (240:2.2cm) {};
  \node[circle,draw=black!100,fill=black!20,inner sep=7pt,scale=0.4]  (n6) at (300:2.2cm) {};

 \node[circle,draw=black!100,fill=black!20,inner sep=2pt,scale=0.5] (n7) at (0:4cm) {$1$};
  \node[circle,draw=black!100,fill=black!20,inner sep=2pt,scale=0.5]  (n8) at (60:4cm) {$2$};
  \node[circle,draw=black!100,fill=black!20,inner sep=2pt,scale=0.5]  (n9) at (120:4cm) {$3$};
  \node[circle,draw=black!100,fill=black!20,inner sep=2pt,scale=0.5]  (n10) at (180:4cm) {$4$};
  \node[circle,draw=black!100,fill=black!20,inner sep=2pt,scale=0.5]  (n11) at (240:4cm) {$5$};
  \node[circle,draw=black!100,fill=black!20,inner sep=2pt,scale=0.5]  (n12) at (300:4cm) {$6$};

   \foreach \from/ \to in {n1/n2, n2/n3,n3/n4, n4/n5,n5/n6, n1/n6,n1/n7, n2/n8,n3/n9, n4/n10, n5/n11, n6/n12}
    \draw (\from) ->(\to);

 \draw (6:3.1cm) node[scale=0.5] {$1$};
 \draw (66:3.1cm) node[scale=0.5] {$1$};
 \draw (114:3.1cm) node[scale=0.5] {$\frac{1}{2}$};
 \draw (171:3.1cm) node[scale=0.5] {$1$};
\draw (248:3.1cm) node[scale=0.5] {$1$};
 \draw (294:3.1cm) node[scale=0.5] {$\frac{1}{2}$};
 
  \draw (30:2.2cm) node[scale=0.5] {$2$};
  \draw (90:2.2cm) node[scale=0.5] {$\frac{3}{2}$};
  \draw (150:2.2cm) node[scale=0.5] {$\frac{3}{2}$};
  \draw (210:2.2cm) node[scale=0.5] {$2$};
  \draw (270:2.2cm) node[scale=0.5] {$\frac{3}{2}$};
  \draw (330:2.2cm) node[scale=0.5] {$\frac{3}{2}$};
\end{tikzpicture} &

\includegraphics[width=4.5cm]{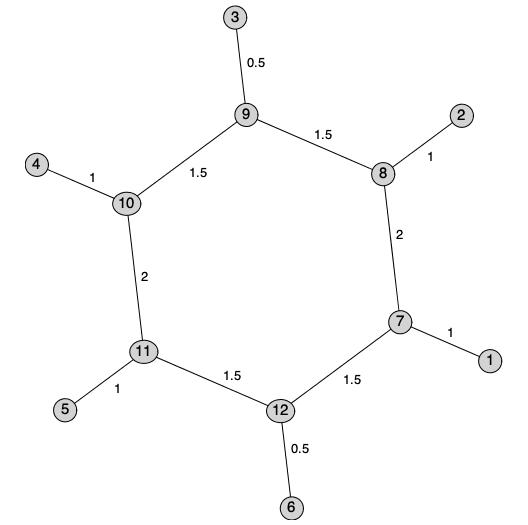}

\end{tabular}
\caption{}\label{maple1}
\end{figure}

\begin{figure}
\centering
\begin{tabular}{cc}
\begin{tikzpicture}
  [scale=0.45,auto=left]
   \node[circle,draw=black!100,fill=black!20,inner sep=7pt,scale=0.4] (n1) at (0:2.2cm) {};
  \node[circle,draw=black!100,fill=black!20,inner sep=7pt,scale=0.4]  (n2) at (72:2.2cm) {};
  \node[circle,draw=black!100,fill=black!20,inner sep=7pt,scale=0.4]  (n3) at (144:2.2cm) {};
  \node[circle,draw=black!100,fill=black!20,inner sep=7pt,scale=0.4]  (n4) at (216:2.2cm) {};
  \node[circle,draw=black!100,fill=black!20,inner sep=7pt,scale=0.4]  (n5) at (288:2.2cm) {};

 \node[circle,draw=black!100,fill=black!20,inner sep=2pt,scale=0.5] (n6) at (0:4cm) {$1$};
  \node[circle,draw=black!100,fill=black!20,inner sep=2pt,scale=0.5]  (n7) at (72:4cm) {$6$};
  \node[circle,draw=black!100,fill=black!20,inner sep=2pt,scale=0.5]  (n8) at (144:4cm) {$5$};
  \node[circle,draw=black!100,fill=black!20,inner sep=2pt,scale=0.5]  (n9) at (216:4cm) {};
  \node[circle,draw=black!100,fill=black!20,inner sep=2pt,scale=0.5]  (n10) at (288:4cm) {$2$};

  \node[circle,draw=black!100,fill=black!20,inner sep=2pt,scale=0.5]  (n11) at (-4.8,-1.4) {$4$};
  \node[circle,draw=black!100,fill=black!20,inner sep=2pt,scale=0.5]  (n12) at (-3,-4) {$3$};

   \foreach \from/ \to in {n1/n2, n2/n3,n3/n4, n4/n5,n5/n1, n1/n6,n2/n7, n3/n8, n4/n9, n5/n10,n9/n11,n9/n12}
    \draw (\from) ->(\to);

 \draw (6:3.1cm) node[scale=0.5] {$2$};
\draw (78:3.1cm) node[scale=0.5] {$1$};
\draw (150:3.1cm) node[scale=0.5] {$1$};
 \draw (222:3.1cm) node[scale=0.5] {$\frac{1}{2}$};
\draw (294:3.1cm) node[scale=0.5] {$1$};

  \draw (36:2.2cm) node[scale=0.5] {$3$};
  \draw (108:2.2cm) node[scale=0.5] {$2$};
  \draw (180:2.2cm) node[scale=0.5] {$1$};
  \draw (252:2.2cm) node[scale=0.5] {$2$};
  \draw (324:2.2cm) node[scale=0.5] {$3$};
  \draw (-3.3,-3.2) node[scale=0.5] {$\frac{5}{2}$};
  \draw (-4,-2.2) node[scale=0.5] {$\frac{3}{2}$};

\end{tikzpicture} &
\includegraphics[width=4.5cm]{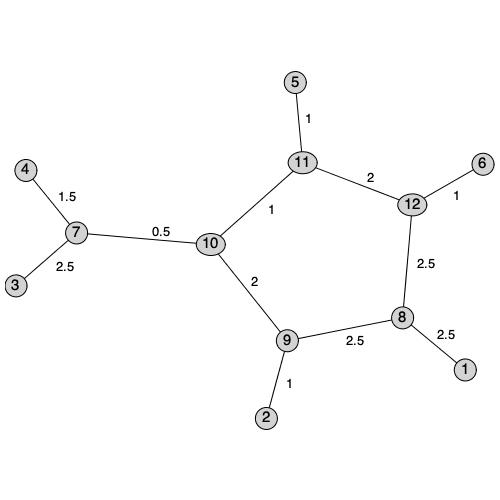}
\\
(a) & (b)
\end{tabular}
\caption{}\label{maple2}
\end{figure}

\begin{figure}
\centering
\begin{tabular}{cc}
\begin{tikzpicture}
  [scale=0.5,auto=left]
   \node[circle,draw=black!100,fill=black!20,inner sep=7pt,scale=0.4] (n1) at (45:2.2cm) {};
  \node[circle,draw=black!100,fill=black!20,inner sep=7pt,scale=0.4]  (n2) at (135:2.2cm) {};
  \node[circle,draw=black!100,fill=black!20,inner sep=7pt,scale=0.4]  (n3) at (225:2.2cm) {};
  \node[circle,draw=black!100,fill=black!20,inner sep=7pt,scale=0.4]  (n4) at (315:2.2cm) {};

   \node[circle,draw=black!100,fill=black!20,inner sep=2pt,scale=0.5] (n5) at (45:4cm) {$6$};
  \node[circle,draw=black!100,fill=black!20,inner sep=2pt,scale=0.5]  (n6) at (135:4cm) {$1$};
  \node[circle,draw=black!100,fill=black!20,inner sep=7pt,scale=0.4]  (n7) at (225:4cm) {};
  \node[circle,draw=black!100,fill=black!20,inner sep=2pt,scale=0.5]  (n8) at (315:4cm) {$5$};
 
  \node[circle,draw=black!100,fill=black!20,inner sep=2pt,scale=0.5]  (n9) at (0,-5) {$4$};
  \node[circle,draw=black!100,fill=black!20,inner sep=7pt,scale=0.4]  (n10) at (-4.2,-2.7) {};
  \node[circle,draw=black!100,fill=black!20,inner sep=2pt,scale=0.5]  (n11) at (-4,-4.7) {$2$};
  \node[circle,draw=black!100,fill=black!20,inner sep=2pt,scale=0.5]  (n12) at (-5,-1) {$3$};

 \foreach \from/ \to in {n1/n2, n2/n3,n3/n4, n4/n1, n1/n5, n2/n6, n3/n7, n4/n8,n7/n9, n7/n10, n10/n11,n10/n12}
    \draw (\from) ->(\to);

\draw (0:1.8cm) node[scale=0.5] {$2$};
\draw (90:1.8cm) node[scale=0.5] {$2$};
\draw (180:1.8cm) node[scale=0.5] {$2$};
\draw (270:1.8cm) node[scale=0.5] {$2$};

\draw (50:3cm) node[scale=0.5] {$1$};
\draw (130:3cm) node[scale=0.5] {$3$};
\draw (230:3cm) node[scale=0.5] {$1$};
\draw (310:3cm) node[scale=0.5] {$3$};

\draw (-1,-3.8) node[scale=0.5] {$1$};
\draw (-3.5,-2.4) node[scale=0.5] {$\frac{1}{2}$};
\draw (-3.9,-3.7) node[scale=0.5] {$\frac{5}{2}$};
\draw (-4.4,-1.8) node[scale=0.5] {$\frac{3}{2}$};

\end{tikzpicture} &
 \includegraphics[width=5cm]{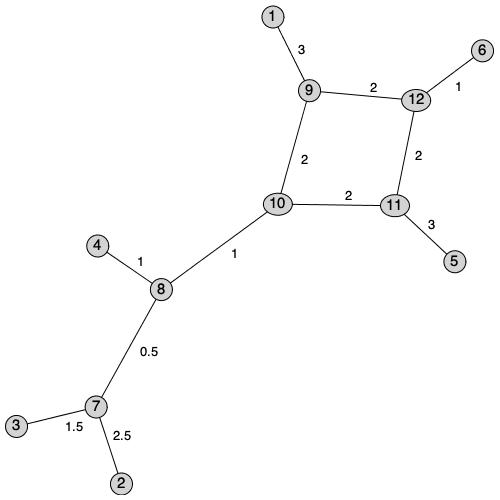}
\\
(a) & (b)
\end{tabular}
\caption{}\label{maple3}
\end{figure}

\begin{figure}
\centering
\begin{tabular}{cc}
\begin{tikzpicture}
  [scale=0.45,auto=left]
   \node[circle,draw=black!100,fill=black!20,inner sep=6pt,scale=0.4] (n1) at (0:2.2cm) {};
  \node[circle,draw=black!100,fill=black!20,inner sep=6pt,scale=0.4]  (n2) at (72:2.2cm) {};
  \node[circle,draw=black!100,fill=black!20,inner sep=6pt,scale=0.4]  (n3) at (144:2.2cm) {};
  \node[circle,draw=black!100,fill=black!20,inner sep=6pt,scale=0.4]  (n4) at (216:2.2cm) {};
  \node[circle,draw=black!100,fill=black!20,inner sep=6pt,scale=0.4]  (n5) at (288:2.2cm) {};

 \node[circle,draw=black!100,fill=black!20,inner sep=6pt,scale=0.4] (n6) at (0:3.5cm) {};
  \node[circle,draw=black!100,fill=black!20,inner sep=6pt,scale=0.4]  (n7) at (72:3.5cm) {};
  \node[circle,draw=black!100,fill=black!20,inner sep=6pt,scale=0.4]  (n8) at (144:3.5cm) {};
  \node[circle,draw=black!100,fill=black!20,inner sep=1pt,scale=0.5]  (n9) at (216:3.5cm) {$11$};
  \node[circle,draw=black!100,fill=black!20,inner sep=6pt,scale=0.4]  (n10) at (288:3.5cm) {};

  \node[circle,draw=black!100,fill=black!20,inner sep=2pt,scale=0.5]  (n11) at (-10:5cm) {$7$};
  \node[circle,draw=black!100,fill=black!20,inner sep=2pt,scale=0.5]  (n12) at (10:5cm) {$6$};

  \node[circle,draw=black!100,fill=black!20,inner sep=2pt,scale=0.5]  (n13) at (55:5cm) {$5$};
  \node[circle,draw=black!100,fill=black!20,inner sep=2pt,scale=0.5]  (n14) at (72:5cm) {$4$};
    \node[circle,draw=black!100,fill=black!20,inner sep=2pt,scale=0.5]  (n15) at (89:5cm) {$3$};
  
   \node[circle,draw=black!100,fill=black!20,inner sep=2pt,scale=0.5]  (n16) at (134:5cm) {$2$};
   \node[circle,draw=black!100,fill=black!20,inner sep=2pt,scale=0.5]  (n17) at (154:5cm) {$1$};
   
     \node[circle,draw=black!100,fill=black!20,inner sep=1pt,scale=0.5]  (n18) at (278:5cm) {$10$};
     \node[circle,draw=black!100,fill=black!20,inner sep=4pt,scale=0.7]  (n19) at (298:5cm) {};      

     \node[circle,draw=black!100,fill=black!20,inner sep=2pt,scale=0.5]  (n20) at (310:6cm) {$8$};      
     \node[circle,draw=black!100,fill=black!20,inner sep=2pt,scale=0.5]  (n21) at (292:6.2cm) {$9$};

   \foreach \from/ \to in {n1/n2, n2/n3,n3/n4, n4/n5,n5/n1,n1/n6,n6/n11, n6/n12,n2/n7,n3/n8,n4/n9,n5/n10,n7/n13,n7/n14,n7/n15,n8/n16,n8/n17,n10/n18, n10/n19, n19/n20, n19/n21}
    \draw (\from) ->(\to);

 \draw (-3,3) node[scale=0.5] {$2$};
 \draw (-3.7,1.9) node[scale=0.5] {$1$};
 \draw (0.4,4) node[scale=0.5] {$\frac12$};
 \draw (1.1,4.1) node[scale=0.5] {$\frac12$};
 \draw (2,3.4) node[scale=0.5] {$\frac12$};
 \draw (4.2,0.8) node[scale=0.5] {$\frac34$};
 \draw (4.2,-0.8) node[scale=0.5] {$\frac74$};
 \draw (2.85,0.35) node[scale=0.5] {$\frac14$};
 \draw (-2.4,-1.45) node[scale=0.5] {$1$};
  \draw (-2.1,1.8) node[scale=0.5] {$2$};
 \draw (0.6,2.7) node[scale=0.5] {$\frac32$};
 \draw (1.05,-2.7) node[scale=0.5] {$1$};
 \draw (0.7,-4.1) node[scale=0.5] {$5$};
 \draw (1.85,-3.75) node[scale=0.5] {$1$};
  \draw (3.1,-4.25) node[scale=0.5] {$1$};
 \draw (2.15,-5.1) node[scale=0.5] {$1$};
 
 \draw (36:2.1cm) node[scale=0.5] {$2$};
  \draw (108:2.1cm) node[scale=0.5] {$1$};
  \draw (180:2.1cm) node[scale=0.5] {$2$};
  \draw (252:2.1cm) node[scale=0.5] {$1$};
  \draw (324:2.1cm) node[scale=0.5] {$1$};

\end{tikzpicture} &
\includegraphics[width=5cm]{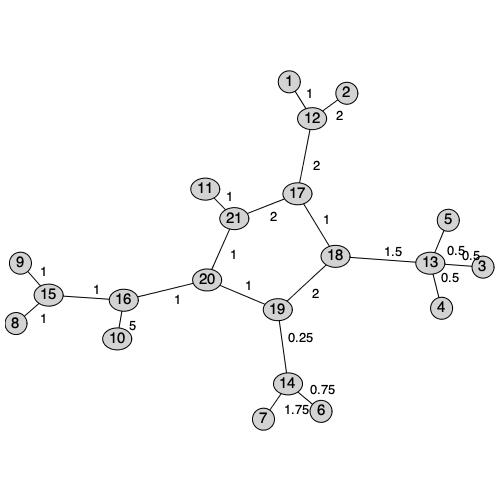}
\\
(a) & (b)
\end{tabular}
\caption{}\label{maple4}
\end{figure}

\section{Conclusions and research directions}

In this paper we present an algorithm, running in time $O(n^4)$, such that, taken as input a distance matrix $D$, firstly returns if $D$ is or not realizable by a tree or a graph of genus 1 and, in case of affirmative answer, it reconstruct the graph itself.
Furthermore, we show that the realization obtained with our algorithm is optimal.

This algorithm concerns the use of the well-known processes of compaction and reduction, even though that of compaction has been slightly modified in such a way to be performed along all the indices of a given distance matrix $D$; this approach leads to the definition of the compaction vector of the matrix $D$ itself.

Moreover, in contrast with other papers, in which compaction and reduction are applied in algorithms, here, we never analyze the behaviour of edges in the graph we are constructing (see, for example, the ``Fusion" step in \cite{V}) or suppose to already have a non optimal realization to transform in an optimal one (as for example in \cite{HHMM}). 
As a matter of fact, looking at the algorithm in the ancillary file \texttt{algorithm.mw}, 
we can see that, once all compaction vectors and reduction matrices are computed, the idea is to construct the weighted adjacency matrix $W$ of the graph which realizes the distance matrix $D$ in input. To this aim, at each step $i$ a block matrix is built, then all blocks are glued together. Extra conditions on $W$ are imposed at the last step, depending on if  the distance matrix $M(t)$, at the last step, is realizable by a tree or by a graph of genus 1.

The conditions given by Theorem \ref{n-agono}, to check if a distance matrix has a realization by a graph of genus 1, lead us to suggest two different directions of research.
The first one is a natural generalization of  Theorem \ref{n-agono},  seeing  if it possible to state a similar result for a realization by a graph of genus $g$, with $g\geq 2$. Even for the case of small $g$, such conditions would be  very interesting also in terms or the following, our second, research direction.

Consider the formula of Theorem \ref{n-agono}. It means that $d_{i\pi^s(i)}$ is equal to the minimum between $d_{i\pi(i)}+d_{\pi(i)\pi^2(i)}+\cdots + d_{\pi^{s-1}(i)\pi^s(i)}$  and $d_{\pi^s(i)\pi^{s+1}(i)}+\cdots + d_{\pi^{n-1}(i)i}$. Then the minimum is attained at least twice among the terms
\[
\begin{array}{c}
d_{i\pi^s(i)},\\ d_{i\pi(i)}+d_{\pi(i)\pi^2(i)}+\cdots + d_{\pi^{s-1}(i)\pi^s(i)},\\ d_{\pi^s(i)\pi^{s+1}(i)}+\cdots + d_{\pi^{n-1}(i)i}.
\end{array}
\]
In terms of Tropical Mathematics (\cite{MS}) this means that the entries of the distance matrix are a zero of the tropical polynomial

\begin{equation}\label{tropngon}
p_{is}:=d_{i\pi^s(i)}\oplus d_{i\pi(i)}\otimes \cdots \otimes d_{\pi^{s-1}(i)\pi^s(i)} \oplus d_{\pi^s(i)\pi^{s+1}(i)} \otimes \cdots \otimes d_{\pi^{n-1}(i)i} 
\end{equation}
Hence the $p_{is}$'s are exactly the tropical equations for the space of weighted $n-$cycles.

The next step is that of understanding how to modify these equations to be satisfied by a distance matrix which has a realization by a graph of genus 1. Reaching this goal, will give us, not only a complete characterization, tropically speaking, of such distance matrices, but also the tropical equations for the moduli space $\mathcal{M}_1$ of elliptic tropical curves. For the same reason, if we find conditions on distance matrices realizable by graphs of genus $g$, and these conditions can be expressed in terms of tropical operations, then we could have a description of moduli space $\mathcal{M}_g$ of plane tropical curves of genus $g$.

Some preliminary results, in this direction, show, for example that if $D$ has a realization by an $n-$cycle $C$ with only one pendant edge for each node of $C$, then its entries are a zero of the following kind of polynomials
\begin{equation}\label{newpoly}
\begin{split}
\tilde{p}_{is}:= &\left[ \left(\bigotimes_{j\not=i,\pi^s(i)}2 a_j \right)\otimes d_{i\pi^s(i)} \right] \oplus \\
\oplus & \left[  \left(\bigotimes_{j=\pi^{s+1}(i)}^{\pi^{n-1}(i)}2 a_j \right) \otimes d_{i\pi(i)}\otimes \cdots \otimes d_{\pi^{s-1}(i)\pi^s(i)} \right]  \oplus \\
\oplus & \left[  \left(\bigotimes_{j=\pi^{1}(i)}^{\pi^{s-1}(i)}2 a_j \right) \otimes d_{\pi^s(i)\pi^{s+1}(i)} \otimes \cdots \otimes d_{\pi^{n-1}(i)i} \right]
\end{split}
\end{equation}

\noindent where the $a_i$ are the entries of the compaction vector of $D$.  
It is easy to check that these polynomials are homogeneous, while the ones in (\ref{tropngon}) are not. Moreover, if $D$ has a realization by an $n-$cycle then its entries are a zero also of this new set of polynomials, giving the hope to use them for a kind of ``four point condition'' for graph of genus 1 and then characterize $\mathcal{M}_1$.

\end{document}